\documentclass[12pt,reqno]{amsart}
\usepackage{amsmath,amssymb,amsfonts}
\usepackage[dvips]{graphicx,color}
\usepackage{mathrsfs}
\usepackage[abbrev]{amsrefs}

\usepackage{latexsym}
\usepackage{delarray} 

\newtheorem{thm}{Theorem}[section]

\newtheorem{prop}[thm]{Proposition}
\newtheorem{lem}[thm]{Lemma}

\newtheorem{definition}[thm]{Definition}

\numberwithin{equation}{section}
\numberwithin{thm}{section}

\title[Besov spaces]{
Besov spaces on open sets}
\author[T. Iwabuchi, T. Matsuyama, 
K. Taniguchi]{Tsukasa Iwabuchi, Tokio Matsuyama 
and Koichi Taniguchi}
\address{ 
Tsukasa Iwabuchi \endgraf 
Department of Mathematics \endgraf 
Osaka City University \endgraf
3-3-138 Sugimoto, Sumiyoshi-ku \endgraf
Osaka 558-8585 \endgraf
Japan}
\email{iwabuchi@sci.osaka-cu.ac.jp}
\address{ 
Tokio Matsuyama \endgraf 
Department of Mathematics \endgraf 
Chuo University \endgraf 
1-13-27, Kasuga, Bunkyo-ku \endgraf 
Tokyo 112-8551 \endgraf 
Japan}
\email{tokio@math.chuo-u.ac.jp} 
\address{ 
Koichi Taniguchi \endgraf 
Department of Mathematics \endgraf 
Chuo University \endgraf 
1-13-27, Kasuga, Bunkyo-ku \endgraf 
Tokyo 112-8551 \endgraf 
Japan} 
\email{koichi-t@gug.math.chuo-u.ac.jp} 
\thanks{
 The first author was supported by 
 Grant-in-Aid for Young
 Scientists Research (B) (No. 25800069), 
 Japan Society for the Promotion of Science.
 The second author was supported by 
Grant-in-Aid for Scientific 
Research (C) (No. 15K04967), 
Japan Society for the Promotion of Science. 
}

\keywords{Besov spaces, Schr\"odinger operators, potential of Kato class}

\begin{document}

\footnote[0]
{2010 {\it Mathematics Subject Classification.} 
Primary 30H25; Secondary 81Q10, 46F05;}

\begin{abstract}
This paper is devoted to 
giving  definitions of 
Besov spaces
on an arbitrary open set of $\mathbb R^n$ 
via the spectral theorem 
for the Schr\"odinger operator with the Dirichlet boundary condition. 
The crucial point is to introduce some 
test function spaces on $\Omega$. 
The fundamental properties 
of Besov spaces are also shown, such as 
embedding relations and duality,
etc. Furthermore, 
the isomorphism relations are established among the Besov spaces 
in which regularity of functions is measured by 
the Dirichlet Laplacian and the Schr\"odinger operators. 
\end{abstract}

\maketitle


\section{Introduction}

In 1959--61 Besov introduced the Besov spaces in his papers \cite{Besov-1959,Besov-1961}. 
There are a lot of literatures on characterization of Besov spaces, 
and we refer to the books of  
Triebel \cite{Triebel_1983,Triebel_1992,Triebel_2006} for history of Besov spaces. 
It was by Peetre that the Fourier transform was employed 
to study the Besov spaces on $\mathbb{R}^n$ 
(see \cite{Pee-1967,Pee-1975,Pee_1976}, 
and also Frazier and Jawerth
\cite{FraJaw-1985,FraJaw-1990
}).  
On a general domain, if the boundary is bounded and smooth, 
the theory of Besov spaces is well established by extending functions on the domain to 
those on $\mathbb R^n$. 
Otherwise, the situation is quite different 
as is indicated in previous studies (see e.g. \cite{Triebel_1983,Triebel-2002}), 
and there appear to be considerable difficulties to construct such theory. 
\\

Let $\Omega$ be an open set of $\mathbb{R}^n$ with $n\geq 1$.
Our aim is to define 
the Besov spaces on $\Omega$
based on the spectral theory
by referring to Peetre's idea. 
If the boundary $\partial \Omega$ of $\Omega$
is smooth, then
some basic notions are available; 
the restriction method of the function on 
$\mathbb R^n$ to $\Omega$, 
the zero extension to the outside of $\Omega$, 
and certain intrinsic characterization 
(see \cite{Mura-1973,Ryc-1998,Ryc-1999,Triebel_1978,Triebel_1983,Triebel_1992,
Triebel-2002,Triebel_2006,TriWin-1996}). 
Recently, 
Bui, Duong and Yan introduced 
some test function spaces
to define
the Besov spaces $\dot B^s_{p,q}$ 
on an arbitrary open set, 
where 
$s,p$ and $q$ satisfy 
$|s|<1$ and $1 \leq p,q \leq \infty$ 
(see \cite{BuDuYa-2012}).
They also proved the equivalence relation 
among 
the Besov spaces generated by 
the Laplacian and some operators, 
including the Schr\"odinger operators, 
on the whole space $\mathbb R^n$, $n \geq 3$ 
with some additional conditions such as H\"older continuity 
for the kernel of semi-group generated by them. 
As to the results on the Besov spaces generated by
the elliptic operators on manifolds, or Hermite operators,
we refer to 
\cite{BuDuYa-2012,BuPaTa-1996,BuPaTa-1997,BuDu-2015,BenZhe-2010,DeSh-1993,KePe-2015} 
and the references therein. 
To the best of our knowledge, 
it is 
necessary to impose some smoothness assumptions on the boundary $\partial \Omega$ 
in order to define the Besov spaces $B^s_{p,q}$ and $\dot B^s_{p,q}$ 
with all indices $s,p,q$ satisfying 
$ s \in \mathbb R$ and $ 1 \leq p,q \leq \infty$. \\

In this paper we shall define the Besov spaces 
$B^s_{p,q}$ and $\dot B^s_{p,q}$ 
generated by 
the Schr\"odinger operator $-\Delta + V$ with the Dirichlet boundary condition for 
all indices $s, p,q$ {\em without any 
geometrical and smoothness assumption on the boundary 
$\partial\Omega$}, 
and shall 
prove the fundamental properties such as embedding relations and 
lifting, etc. 
Furthermore, regarding the Besov spaces generated by 
the Dirichlet Laplacian as the standard one,
and adopting the potential $V$ 
belonging to the Lorentz space $L^{\frac{n}{2},\infty} (\Omega)$, 
we shall establish the equivalence relation between 
the Besov spaces generated by the Dirichlet Laplacian $-\Delta_{|D} $ 
and Schr\"odinger operator $-\Delta _{|D} + V$. 
The motivation of the study of such properties and equivalence relation comes 
from their applications to partial differential equations, 
and one can consult the papers of 
D'Ancona and Pierfelice (see \cite{DP-2005}), Georgiev and Visciglia (see \cite{GV-2003})
and Jensen and Nakamura (see \cite{JN-1994,JN-1995}). 
\\

Let us consider the Schr\"odinger operator 
\[
- \Delta 
+ V(x) 
= -\sum^{n}_{j=1} \frac{\partial^2}{\partial x^2_j}+V(x ) 
\]
on an arbitrary open set $\Omega$
with the Dirichlet boundary condition, 
where 
$V(x)$ is a real-valued measurable function on $\Omega$. 
In this paper we adopt potentials whose 
negative parts belong to the Kato class. 
More precisely, let us assume that the 
potential $V$ 
satisfies 
\begin{equation}\label{1104-1}
V = V_{+} - V_-, \quad 
V_{\pm} \geq 0, \quad 
V_+ \in  L^1_{\rm loc} (\Omega) 
\text{ and } 
V_- \in K_n (\Omega) . 
\end{equation}
Here, 
the negative part $V_-$ of $V$ 
  is said to belong to the Kato class 
$K_{n}( \Omega)$ if $V_-$ satisfies 
\begin{align}\notag 
\left\{
\begin{aligned}
	&\lim_{r \rightarrow 0} \sup_{x \in \Omega} \int_{\Omega \cap \{|x-y|<r\}} 
	   \frac{|V_-(y)|}{|x-y|^{n-2}} \,dy = 0, &n\ge 3, \\
	&\lim_{r \rightarrow 0} \sup_{x \in \Omega} \int_{\Omega \cap \{|x-y|<r\}} 
	   \log (|x-y|^{-1})|V_-(y)| \,dy = 0, &n=2, \\
	&\sup_{x \in \Omega}\int_{\Omega \cap \{|x-y|<1\}} |V_-(y)| \,dy <\infty, &n=1. 
	\end{aligned}\right.
\end{align}
Then $-\Delta  + V$ has a self-adjoint realization on $L^2 (\Omega)$ 
(see Lemma~\ref{lem:s.a.} in 
appendix~\ref{App:AppendixA}). 
Throughout this paper, we use the following notation: 
\\

\noindent 
{\bf Notation. }
{\it 
We denote by $A_V$ the self-adjoint realization of $-\Delta + V$ with the domain 
\begin{equation}\label{EQ:ID}
\mathcal D (A_V) 
= \big\{ f \in H^1_0(\Omega)  \, \big| \, 
    \sqrt{V_+} f, \, A_V f \in L^2 (\Omega) \big\},
\end{equation}
where 
$H^1_0(\Omega)$ is the completion of $C_0^\infty (\Omega)$ 
with respect to the norm
$$
\| f \|_{L^2 (\Omega)} + \| \nabla f \|_{L^2 (\Omega)} .
$$
}  

By the spectral theorem
there exists a spectral resolution 
$\{ E_{A_V}(\lambda) \}_{\lambda \in \mathbb{R}}$
of the identity and $A_V$ is written as
\[
	A_V = \int^{\infty}_{-\infty} \lambda \, d E_{A_V}(\lambda).  
\]
For a Borel measurable function $\phi(\lambda)$ on $\mathbb R$, 
$\phi(A_V)$ is defined by letting
\[
	\phi(A_V) = \int^{\infty}_{-\infty} \phi (\lambda) \, d E_{A_V}(\lambda)  
\]
with the domain 
\[
\mathcal D ( \phi(A_V)) =\left\{ f \in L^2 (\Omega) \, \Big| \, 
\int_{-\infty}^\infty 
|\phi (\lambda)|^2 d \| E_{A_V} (\lambda) f \|_{L^2 (\Omega)} ^2 < \infty
\right\}. 
\]
Due to such a spectral resolution, 
we can define the Sobolev spaces $H^s(A_V) $ by 
letting 
\begin{equation}\label{Hs}
H^s(A_V) 
= 
\left\{ f \in L^2 (\Omega) \, \big| \, 
   (I+A_V ) ^{\frac{s}{2}} f \in L^2 (\Omega) 
\right\} \quad \text{for } s \geq 0. 
\end{equation}
Then, 
the regularity and boundary value of functions in $H^s (A_V)$ are 
determined by $A_V$.  
Hereafter, we call $H^s (A_V)$ 
{\em the Sobolev spaces by $A_V$-regularity}. 
In particular case $\Omega = \mathbb R^n$ and $V = 0$, 
the Sobolev spaces defined in \eqref{Hs} coincide with 
the Bessel-potential spaces defined via the Fourier transform.

\vskip3mm 

We shall apply the above characterization of $H^s (A_V)$ 
to those of the inhomogeneous and homogeneous Besov spaces (see Theorem \ref{thm:1} below). 
For the Besov spaces by this characterization, 
we  obtain fundamental properties of the spaces 
(see Propositions \ref{thm:2}--\ref{thm:7} below) 
and 
find a sufficient 
condition on the integrability of $V$ such that 
the isomorphism holds between the Besov spaces by $A_0$ and $A_V$-regularity 
(see Proposition \ref{thm:6} below). 
{\em It should be noted that 
our framework on open sets $\Omega$ of $\mathbb R^n$ 
is the most general setting. The crucial point is to introduce 
test function spaces on $\Omega$.}
\\

Let us recall the definitions of the 
test function spaces on $\mathbb R^n$ and 
the classical Besov spaces, i.e., 
spaces when 
$\Omega  = \mathbb R^n$ and $V = 0$. 
It is well known that the inhomogeneous 
Besov spaces and homogeneous ones are 
characterized as subspaces of 
$\mathcal{S}'(\mathbb R^n)$ and $\mathcal{Z}'(\mathbb R^n)$ 
by the Littlewood-Paley 
dyadic decomposition of the spectrum of $\sqrt {-\Delta}$, namely, 
$B^s_{p,q} $ and $\dot B^s_{p,q} $ consist 
of all $f \in \mathcal S' (\mathbb R^n)$ and 
$\mathcal Z '(\mathbb R^n)$ such that 
\begin{gather}
\notag 
\| f \|_{B^s_{p,q}} =
    \big\| \mathcal F^{-1} \psi (  | \xi | ) \mathcal F f \big\|_{L^p  (\mathbb R^n)} + 
    \big\| \big\{ 2^{sj} \| \mathcal F^{-1} \phi (2^{-j}|\xi| ) \mathcal F f \|_{L^p  (\mathbb R^n)} 
    \big\} _{j \in \mathbb N} \big\|_{\ell ^q (\mathbb N)} < \infty , 
\\ \notag 
\| f \|_{\dot B^s_{p,q}} = 
    \big\| \big\{ 2^{sj} \| \mathcal F^{-1} \phi (2^{-j} |\xi| ) \mathcal F f \|_{L^p  (\mathbb R^n)} 
    \big\} _{j \in \mathbb Z} \big\|_{\ell ^q (\mathbb Z)} < \infty , 
\end{gather}
respectively, for some smooth functions $\psi , \phi$ with compact supports. 
Here $\mathcal S^\prime (\mathbb R^n)$ is
the space of the tempered distributions on 
$\mathbb R^n$, which is the 
topological dual 
of the Schwartz space 
$\mathcal{S}(\mathbb R^n)$. The space 
$\mathcal S (\mathbb R^n)$ consists of rapidly decreasing functions 
equipped with the family 
of semi-norms 
\begin{equation}\label{1105-5}
\sup _{x \in \mathbb R^n} 
  (1+ |x|^2)^{\frac{M}{2}} \sum _{|\alpha |\leq M} |\partial_x ^\alpha f (x)|, 
  \quad 
  M = 1,2, \cdots.
\end{equation}
$\mathcal Z'(\mathbb R^n)$ is the dual space of 
$\mathcal Z(\mathbb R^n)$, which is the subspace
of $\mathcal{S}(\mathbb R^n)$ defined by letting
\begin{equation}\label{18-3}
\mathcal Z(\mathbb R^n) 
:= 
\Big\{ f \in \mathcal S (\mathbb R^n) 
\, \Big| \, 
\int_{\mathbb R^n} x^\alpha f(x) dx = 0  \,\, 
\text{ for all } \alpha \in 
( \mathbb N\cup \{ 0 \})^n
\Big\} 
\end{equation}
endowed with the induced topology of 
$\mathcal S (\mathbb R^n)$. 
It is known that 
$\mathcal Z' (\mathbb R^n)$ is characterized by 
the quotient space of $\mathcal S' (\mathbb R^n)$ modulo polynomials, i.e.,
\[
\mathcal Z' (\mathbb R^n) \simeq \mathcal S' (\mathbb R^n)/\mathcal{P}(\mathbb R^n),
\]
where $\mathcal{P}(\mathbb R^n)$ is the set of
all polynomials of $n$ real variables 
(see e.g. \cite{Triebel_1983,Grafakos_2014n}).\\

When $\Omega \not = \mathbb R^n$, a question naturally arises 
what the spaces corresponding to $\mathcal S' (\mathbb R^n)$ 
and $\mathcal Z'(\mathbb R^n)$ are.  
We  introduce a kind of spaces as $\mathcal X_V '(\Omega)$ 
and $\mathcal Z_V' (\Omega)$ in \S\S 2.1. 
There we will encounter with two problems 
in the formulations:
\begin{enumerate}
\item[(a)] 
To handle the neighborhood of zero spectrum in 
 the definition of the homogeneous Besov spaces;

\item[(b)] 
To develop the dyadic resolution of identity operators 
on our spaces $\mathcal X'_V(\Omega)$ 
and $\mathcal Z_V ' (\Omega)$; dyadic resolution lifted from $L^2 (\Omega)$.  

\end{enumerate}

Let us explain the problem 
(a). Looking at the definition \eqref{18-3} of 
$\mathcal Z (\mathbb R^n)$, 
one understands that the low frequency part 
of $f$ is treated by 
\begin{equation}\label{EQ:low}
\int_{\mathbb R^n} x^\alpha f(x) dx = 0  
\quad \text{for any  } \alpha \in (\mathbb N \cup \{ 0 \})^n.
\end{equation}
However, when $\Omega\ne \mathbb R^n$ 
it seems 
difficult to get an idea corresponding to \eqref{EQ:low}. To overcome 
this difficulty, instead of \eqref{EQ:low}, 
we propose 
\begin{equation}\label{18-8n}
\sup_{j \leq 0} 2^{M |j| } 
\big\| \phi_j(\sqrt{A_V})  f \big\|_{L^1 (\Omega)} < \infty 
 , \quad M = 1,2,\cdots
\end{equation}
in semi-norms $q_{V,M}(\cdot)$ of a test function space $\mathcal{Z}_V(\Omega)$ 
(see \eqref{EQ:qV} below), where we put
\[
\phi_j (\sqrt{A_V}) := \phi (2^{-j} \sqrt{A_V}). 
\] 
This is probably 
 a main novelty in our work. 
The condition \eqref{18-8n} seems one of important ingredients 
to introduce test function spaces not only for Besov spaces but also 
for other spaces of homogeneous type. 
We note that a kind of the problem of zero spectrum does not appear 
in the inhomogeneous Besov spaces,  
and hence, our spaces $\mathcal X_V (\Omega)$ and $\mathcal X_V ' (\Omega)$ 
may be analogous to the test function spaces and their duals
introduced by Kerkyacharian and Petrushev~\cite{KePe-2015} 
(see also Ruzhansky and Tokmagambetov~\cite{RT-toappear} who treat 
$H^s (A_V)$ on a bounded open set, and the operator $A_V$ does not have to be 
self-adjoint). 
\\

We turn to explain the problem 
(b). For the sake of simplicity,
let us consider the case when 
$V = 0$. Clearly, in this case, $A_V$ becomes the Dirichlet Laplacian $A_0$.  
As is well-known, the identity operator is resolved by the dyadic 
decomposition of the spectrum for the Dirichlet Laplacian in $L^2 (\Omega)$, 
namely, 
\begin{equation}\label{18-1n}
I = \psi (A_0) + \sum _{ j \in \mathbb N } \phi_j (\sqrt{A_0}) , 
\end{equation}
which is assured by the spectral theorem, 
where $\psi$ is a smooth function such that 
$$
\psi (\lambda ^2) + \sum _{ j \in \mathbb N} \phi _j (\lambda) = 1 
\quad \text{for any } \lambda \geq 0 . 
$$
Initially, the resolution \eqref{18-1n} holds in $L^2 (\Omega) $, 
and then, it is lifted to the space $\mathcal X_0 ' (\Omega)$.  
This argument is accomplished in Lemma~\ref{lem:decomposition1} below. 
When one considers $ \mathcal Z'_0 (\Omega)$,
\eqref{18-1n} is replaced by 
\begin{equation}\label{18-2n}
I = \sum _{ j \in \mathbb Z} \phi_j (\sqrt{A_0}) . 
\end{equation}
Thanks to these resolutions 
\eqref{18-1n} and \eqref{18-2n}, 
the well known methods in the classical Besov spaces on $\mathbb R^n$ work well also in the present case. The
starting point of this argument is to extend the spectral restriction operators $\phi_j (\sqrt{A_0})$ 
on $L^2 (\Omega)$ 
to those on $L^1 (\Omega)$. There, 
the uniform boundedness on 
 $L^1 (\Omega)$ 
of $\{ \phi_j (\sqrt{A_0}) \}_j$, i.e., 
\begin{equation}\label{1126-1n}
\sup _{j }\| \phi_j (\sqrt{A_0}) \| _{L^1 (\Omega) \to L^1 (\Omega)} < \infty 
\end{equation}
plays a crucial role in proving \eqref{18-1n} in 
$\mathcal{X}_0 '(\Omega)$ and \eqref{18-2n} in $\mathcal Z_0 ' (\Omega)$, 
respectively. 
For the proof, see Proposition 
\ref{prop:Lp-bound} in appendix A (see also \cite{IMT-preprint,IMT-ISAAC}). 
Furthermore, 
\eqref{1126-1n} guarantees the independence of the choice of $\{ \phi_j \}_{j \in \mathbb Z} \cup \{ \psi \}$,
when we define spaces $\mathcal X_0 (\Omega) $, $\mathcal X_0 ' (\Omega)$, $\mathcal Z_0 (\Omega)$, 
$\mathcal Z_0 ' (\Omega)$ and Besov spaces defined in \S 2. 
\\

This paper is organized as follows. 
In \S 2, we state a main result on the Besov spaces 
by $A_V$-regularity. 
\S 3 is devoted to stating some 
fundamental properties of Besov spaces. 
In \S 4, we introduce key lemmas and 
fundamental properties of test function spaces on $\Omega $, 
which are essential for our theory. 
\S5--\S9 are devoted to the proof of our results. 
In appendix \ref{App:AppendixA}, we 
show the uniform $L^p$-boundedness of $\phi (\theta A_V)$, 
the self-adjointness of $A_V$ and the pointwise estimate for the kernel of $e^{-tA_V}$ 
which are verified with some modifications 
of our previous work~\cite{IMT-preprint}. 
Finally, we prove in appendix \ref{App:AppendixB} that zero is not an eigenvalue of $A_V$ 
under some smallness condition on the negative part of $V$. 


\section{Statement of results}

In this section we shall state several results 
on the Besov spaces 
by $A_V$-regularity. 
We divide this section into two subsections: 
the introduction of test function spaces, and 
statement of the result.

\subsection{Definitions of test function spaces on $\Omega$}
In this subsection we shall define 
``test function spaces" 
consisting of functions smooth and integrable on $\Omega$ 
 and spaces of a kind of ``tempered distributions'' 
as follows: \\

Let $\phi_0(\cdot) \in C^\infty_0(\mathbb R)$ be a non-negative function on $\mathbb R$ such that  
\begin{equation}
\label{917-1}
{\rm supp \, } \phi _0
\subset \{ \, \lambda \in \mathbb R \, | \, 2^{-1} \leq \lambda \leq 2 \, \}, 
\quad \sum _{ j \in \mathbb Z} \phi_0 ( 2^{-j}\lambda) 
 = 1 
 \quad \text{for } \lambda > 0,  
\end{equation}
and $ \{ \phi_j \}_{j \in \mathbb Z}$ is defined by letting 
\begin{gather} \label{917-2}
\phi_j (\lambda) = \phi_0 (2^{-j} \lambda) 
 \quad \text{for }  \lambda \in \mathbb R . 
\end{gather}

\begin{definition}\label{def:1}

\begin{enumerate}
\item[(i)] {\rm (}Linear topological spaces
$\mathcal X_V (\Omega)$ and $\mathcal X^\prime_V (\Omega)${\rm ).}
Assume that the measurable potential $V$ 
satisfies \eqref{1104-1}. 
Then a linear topological space 
$\mathcal X_V (\Omega)$ is 
defined by letting
\begin{equation}\notag 
\mathcal X_V (\Omega) 
:= \big\{ f \in  L^1 (\Omega) \cap \mathcal D (A_V) 
 \, \big| \, 
    A^{M}_V f \in L^1(\Omega ) \cap \mathcal D (A_V) \text{ for all } M \in \mathbb N 
   \big\} 
\end{equation} 
equipped with the family of semi-norms $\{ p_{V,M}
 (\cdot) \}_{ M = 1 } ^\infty$ 
given by 
\begin{equation}\notag 
p_{V,M}(f) := 
\| f \|_{ L^1(\Omega)} 
+ \sup _{j \in \mathbb N} 2^{Mj} 
  \| \phi_j (\sqrt{A_V}) f \|_{ L^1(\Omega)} . 
\end{equation}
$\mathcal X'_V(\Omega)$ denotes the topological dual of 
$\mathcal X_V (\Omega)$ 
and 
$
{}_{{\mathcal X}_V' }
\langle f , g \rangle _{{\mathcal X}_V}  
$
is the duality pair of  
$f \in \mathcal{X}'_V(\Omega)$
 and 
$g \in \mathcal{X}_V(\Omega)$. 
A sequence $\{ f_N \}_{N = 1} ^\infty$ in $\mathcal X_V ' (\Omega)$ 
is said to converge to $f \in \mathcal X_V ' (\Omega)$ if 
$$
 {}_{{\mathcal X}_V' }
\langle f_N , g \rangle _{{\mathcal X}_V}  
\to  {}_{{\mathcal X}_V' }
\langle f , g \rangle _{{\mathcal X}_V}  
\quad \text{as } N \to \infty 
\quad \text{for any } g \in \mathcal X_V (\Omega). 
$$

\item[(ii)] {\rm (}Linear topological spaces
$\mathcal Z_V (\Omega)$ and $\mathcal Z^\prime_V (\Omega)${\rm ).} 
Assume that the measurable potential $V$ 
satisfies \eqref{1104-1} and 
\begin{gather}
\label{ass:1}
\begin{cases} 
V _- = 0
& \quad \text{if } n = 1,2, 
\\
\displaystyle \sup _{x \in \Omega} 
   \int_{\Omega} \dfrac{|V_- (y)|}{|x-y|^{n-2}} \, dy 
 < \dfrac{\pi^{\frac{n}{2}}}{\Gamma (n/2 -1)}
& \quad \text{if } n \geq 3 . 
\end{cases}
\end{gather}
Then a linear topological space 
$\mathcal Z_V (\Omega)$ is 
defined by letting
\begin{equation}\notag 
\mathcal Z_V (\Omega) 
:= \Big\{ f \in \mathcal X_V (\Omega) 
 \, \Big| \, 
  \sup_{j \leq 0} 2^{ M |j|} 
    \big\| \phi_j \big(\sqrt{ A_V } \big ) f \big \|_{L^1(\Omega)} < \infty 
  \text{ for all } M \in \mathbb N
   \Big\} 
\end{equation}
equipped 
with the family of semi-norms $\{ q_{V,M} (\cdot) \}_{ M = 1}^\infty$ 
given by 
\begin{equation}\label{EQ:qV} 
q_{V,M}(f) := 
\| f \|_{L^1 (\Omega) }
+ \sup_{j \in \mathbb Z} 2^{M|j|} \| \phi_j (\sqrt{A_V}) f \|_{L^1(\Omega)}. 
\end{equation}
$\mathcal Z'_V(\Omega)$ denotes the topological dual of $\mathcal Z_V (\Omega)$ 
and 
${}_{{\mathcal Z}_V' }
\langle f , g \rangle _{{\mathcal Z}_V}$ 
is the duality pair of  
$f \in \mathcal{Z}'_V(\Omega)$
 and 
$g \in \mathcal{Z}_V(\Omega)$. 
A sequence $\{ f_N \}_{N = 1} ^\infty$ in $\mathcal Z_V ' (\Omega)$ 
is said to converge to $f \in \mathcal Z_V ' (\Omega)$ if 
$$
 {}_{{\mathcal Z}_V' }
\langle f_N , g \rangle _{{\mathcal Z}_V}  
\to  {}_{{\mathcal Z}_V' }
\langle f , g \rangle _{{\mathcal Z}_V} 
\quad \text{as } N \to \infty 
\quad \text{for any } g \in \mathcal Z_V (\Omega). 
$$

\item[(iii)] 
{\rm (}Spaces by the Dirichlet Laplacian regularity{\rm )}.  
In particular case $V=0$, we write 
$\mathcal{X}_V(\Omega)$, 
$\mathcal{X}^\prime_V(\Omega)$, 
$\mathcal{Z}_V(\Omega)$ and 
$\mathcal{Z}^\prime_V(\Omega)$ as 
\[
\mathcal{X}_0(\Omega), \quad  
\mathcal{X}^\prime_0(\Omega), \quad 
\mathcal{Z}_0(\Omega) \quad \text{and} \quad 
\mathcal{Z}^\prime_0(\Omega), 
\]
respectively. 
\end{enumerate}

\end{definition}


We notice 
from assumption 
\eqref{ass:1} 
that  $A_V$ is non-negative on $L^2 (\Omega)$ 
and that zero is not an eigenvalue of $A_V$ 
as well as the Dirichlet Laplacian. 
In fact, these results hold for a weaker assumption \eqref{ass:B_1} in appendix \ref{App:AppendixB}. 
We also note that 
assumption \eqref{ass:1} excludes the potential $V$ like 
$$
V(x) = -c |x|^{-2}, 
\quad c > 0 . 
$$
For more details, see the remark after the statement of Proposition \ref{prop:Lp-bound} 
in appendix~A. 
\\

Functions in the Lebesgue spaces are 
regarded as elements in $\mathcal X_V'(\Omega)$ and $\mathcal Z_V'(\Omega)$ 
analogously to the case for $\mathcal S' (\mathbb R^n)$ 
and $\mathcal Z' (\mathbb R^n)$,  respectively. 
Lemma \ref{cor:1} below assures that 
$$
\int_ {\Omega} \big|f (x) \overline{g (x) } \big| \, dx  < \infty  
$$
for any $f \in L^p (\Omega)$, $1 \leq p \leq \infty$, and $g \in \mathcal X_V (\Omega)$ 
($g \in \mathcal Z_V (\Omega) $ resp.). 
So, we define: 

\begin{definition}\label{def:2}
For $f\in L^1(\Omega)+L^\infty(\Omega)$,  
we identify 
$f$ as an element in $\mathcal{X}^\prime_V(\Omega)$
{\rm (}$\mathcal{Z}^\prime_V(\Omega)$ resp.{\rm )}
by letting
\[
{}_{\mathcal{X}'_V} \langle f , g \rangle _{\mathcal{X}_V} 
=  \int_\Omega f(x)\overline{g(x)}\, dx 
\quad \left(
{}_{\mathcal{Z}'_V} \langle f , g \rangle _{\mathcal{Z}_V} 
=  \int_\Omega f(x)\overline{g(x)}\, dx
\quad \mathrm{resp.}\right)
\]
for any $g\in \mathcal{X}_V(\Omega)$ 
{\rm (}$g\in \mathcal{Z}_V(\Omega)$ 
resp.{\rm )}. 
\end{definition}

For a mapping 
$\phi (A_V)$ on $\mathcal X_V (\Omega)$ ($\mathcal Z_V (\Omega)$ resp.), 
we define the dual operator 
of $\phi (A_V)$ on 
$\mathcal X_V' (\Omega)$ ($\mathcal Z_V' (\Omega)$ resp.)
induced naturally from that on $L^2 (\Omega)$. 

\begin{definition}\label{def:3}
\begin{enumerate}
\item[(i)] 
For a mapping $\phi (A_V): \mathcal X_V (\Omega) \to \mathcal X_V (\Omega)$, 
we define $\phi (A_V): \mathcal X'_V (\Omega) \to \mathcal X'_V (\Omega)$ 
by letting 
\begin{equation}\label{901-4}
{}_{{\mathcal X}'_V }
\big\langle \phi (A_V) f , g \big\rangle _{{\mathcal X}_V} 
:= 
{}_{{\mathcal X}'_V } \big\langle f , \phi (A_V) g \big\rangle_{{\mathcal X}_V } 
\quad \text{for all } g\in \mathcal X_V(\Omega). 
\end{equation}

\item[(ii)] 
For a mapping $\phi (A_V) : \mathcal Z_V (\Omega) \to \mathcal Z_V (\Omega)$, 
we define $\phi (A_V) : \mathcal Z'_V (\Omega) \to \mathcal Z'_V (\Omega)$ 
by letting 
\begin{equation}\label{901-5}
{}_{\mathcal{Z}'_V} \big\langle \phi (A_V) f , g \big\rangle _{\mathcal{Z}_V}
:= 
{}_{\mathcal{Z}'_V} \big\langle f , \phi (A_V) g \big\rangle _{\mathcal{Z}_V} 
\quad \text{for all } g\in \mathcal Z_V(\Omega).  
\end{equation}

\end{enumerate}
\end{definition}


It is shown in Lemma \ref{EQ:Frechet} 
below 
that $\mathcal X_V (\Omega)$ and 
$\mathcal Z_V (\Omega)$ are complete, 
and hence, they are Fr\'echet spaces. 
Needless to say,
it is not possible to define an operator 
$\sqrt{A_V}$ if the spectrum of $A_V$ contains 
negative real numbers. 
However, since $\phi_j  (\lambda) = 0$ for $\lambda \leq 0$, 
we define $\phi_j (\sqrt{A_V}) $ as 
$$
\phi_j (\sqrt{A_V}) 
= \int_{0}^{\infty} \phi_j (\sqrt{\lambda}) dE_{A_V} (\lambda).
$$

\medskip 

Let us give a few remarks on properties of 
 $\mathcal X_0(\Omega)$ and $\mathcal Z_0(\Omega)$ as follows: 

\vskip3mm

\begin{itemize}
\item When $\Omega=\mathbb R^n$ 
and $V=0$, the Schwartz space $\mathcal S(\mathbb R^n)$ is contained in 
$\mathcal X_0 (\mathbb R^n)$, and 
the inclusion for tempered distributions are just opposite. Namely, 
it can be readily checked 
from Definition \ref{def:1} that  
\begin{gather}
\label{EQ:INC} 
\mathcal S (\mathbb R^n) 
\hookrightarrow \mathcal X _0 (\mathbb R^n) 
\hookrightarrow \mathcal X' _0 (\mathbb R^n) 
\hookrightarrow \mathcal S' (\mathbb R^n), 
\\ \notag 
\mathcal Z (\mathbb R^n) 
\hookrightarrow \mathcal Z _0 (\mathbb R^n) 
\hookrightarrow \mathcal Z' _0 (\mathbb R^n) 
\hookrightarrow \mathcal Z' (\mathbb R^n), 
\\ \notag 
C_0 ^\infty (\mathbb R^n) \subset \mathcal X_0(\mathbb R^n), \quad  
C_0 ^\infty (\mathbb R^n) \not \subset 
\mathcal Z _0 (\mathbb R^n).
\end{gather}

\item 
When $\Omega=\mathbb R^n$ and $V=0$, the restriction of 
low frequency in the definition \eqref{EQ:qV} 
of $q_{0,M}(f)$ is natural,
since one can show 
that any element 
$f\in \mathcal S (\mathbb R^n)$ 
belongs to 
$\mathcal Z (\mathbb R^n)$ if and only if 
$q_{0,M} (f) < \infty$ for $ M = 1,2, \dots$. 

\item 
When $\Omega \not = \mathbb R^n$,
any $f \in \mathcal X_0 (\Omega)$ 
or $\mathcal Z _0(\Omega)$ satisfies 
\[
f \equiv 0 \quad \text{on }  \partial \Omega,
\]
since $f \in H^1_0 (\Omega)$. 
Hence, 
the condition $p_{0,M} (f) < \infty$ 
not only determines smoothness 
and integrability of $f$ 
but also assures the Dirichlet boundary condition. 
Also, such an $f$ contacts with $\partial\Omega$ of order infinity in the following way:
$$
A_{0}^M f \equiv 0 \quad \text{on }  \partial \Omega  ,
\quad M = 0,1,2,\cdots.
$$
The same assertion holds for $\mathcal X_V (\Omega)$, 
$\mathcal Z_V (\Omega)$ and $A_V$. 

\item 
In order to simplify the argument, 
instead of the polynomial weights 
appearing on semi-norms \eqref{1105-5} 
in $\mathcal S (\mathbb R^n)$, we adopted 
the integrability condition on $f$. 
\end{itemize}

\quad

Based on Definitions 
\ref{def:1}--\ref{def:3}, we establish 
the definition of Besov spaces on an arbitrary 
open set of $\mathbb R^n$ in \S\S 2.2.


\subsection{Statement of the main result} 

We are in this subsection to state the result.
Let $\{ \phi_j \}_{j \in \mathbb Z} \cup \{ \psi \}$ be
the Littlewood-Paley dyadic decomposition, namely, the sequence 
$\{\phi_j\}_{j \in \mathbb Z}$ is defined by \eqref{917-2}, and 
$\psi$ is a smooth function with compact support around the origin. 
Here, we note that if $V$ satisfies assumption \eqref{1104-1}, 
then the spectrum of $A_V$ may admit to be negative.  
It is shown in Lemma~\ref{lem:s.a.} in 
appendix A that
there exists a positive constant $\lambda _0$ such that 
\begin{equation}\label{EQ:A-bound}
A_V \geq - \lambda _0^2I . 
\end{equation}
Then we need to choose  
the function $\psi $ such that 
\begin{gather} \label{917-3}
\psi (\lambda) = 1 \quad 
\text{ for }
\lambda \in [-\lambda _0^2 , 0],
\qquad 
\psi (\lambda^2) 
 + \sum _{ j \in \mathbb N}  \phi_j (\lambda) 
 = 1 
 \quad \text{for } \lambda \geq 0. 
\end{gather}
Based on this choice of $\psi$, let us introduce the definition of the Besov spaces 
by $A_V$-regularity.


\begin{definition}
For $s \in \mathbb R$ and $1 \leq p,q \leq \infty$, 
we define the inhomogeneous and homogeneous Besov spaces as follows{\rm :} 
\begin{enumerate}
\item[(i)] 
 $B^s_{p,q} (A_V) $ is defined by letting 
\begin{equation}\notag
B^s_{p,q} (A_V)
:= \{ f \in \mathcal X'_V(\Omega) 
     \, | \, 
     \| f \|_{B^s_{p,q} (A_V)} < \infty
   \} , 
\end{equation}
where
\begin{equation}\notag 
     \| f \|_{B^s_{p,q} (A_V)} 
       := \| \psi ( A_V) f \|_{L^p (\Omega) } 
          + \big\| \big\{ 2^{sj} \| \phi_j (\sqrt{A_V}) f \|_{L^p (\Omega) }  
                 \big\}_{j \in \mathbb  N}
          \big\|_{\ell^q (\mathbb N)}. 
\end{equation}

\item[(ii)] 
$\dot B^s_{p,q} (A_V) $ is defined by letting
\begin{equation} \notag 
\dot B^s_{p,q}  (A_V)
:= \{ f \in \mathcal Z'_V(\Omega) 
     \, | \, 
     \| f \|_{\dot B^s_{p,q}(A_V)} < \infty 
   \} , 
\end{equation}
where
\begin{equation}\notag 
     \| f \|_{\dot B^s_{p,q}(A_V)} 
       := \big\| \big\{ 2^{sj} \| \phi_j (\sqrt{A_V}) f\|_{L^p(\Omega)}
                 \big\}_{j \in \mathbb Z}
          \big\|_{\ell^q (\mathbb Z)}. 
\end{equation}
\end{enumerate}
\end{definition}

\vskip3mm


Our main result can now be 
formulated in the following way:
\begin{thm}\label{thm:1}
For any $s,p,q$ with $s \in \mathbb R$ and $1 \leq p,q \leq \infty$, the 
following assertions hold{\rm :}

\begin{enumerate}
\item[(i)] 
{\rm(}Inhomogeneous Besov spaces{\rm)} Assume that 
the measurable potential $V$ satisfies \eqref{1104-1}. Then{\rm :} 

\begin{enumerate}
\item[ (i-a)] 
$B^s_{p,q} (A_V)$ is independent 
of the choice of $\{ \psi \} \cup \{ \phi_j\}_{j \in \mathbb N}$ 
satisfying 
\eqref{917-1}, 
\eqref{917-2} and \eqref{917-3}, 
and enjoys the following{\rm :} 
\begin{equation}\label{902-1}
\mathcal X_V(\Omega) 
\hookrightarrow B^s_{p,q } (A_V)
\hookrightarrow \mathcal X'_V(\Omega).
\end{equation} 

\item[ (i-b)]
$B^s_{p,q}(A_V)$ 
is a Banach space.
\end{enumerate}

\item[(ii)]
{\rm (}Homogeneous Besov spaces{\rm )} 
Assume that 
the measurable potential $V $ satisfies \eqref{1104-1} and \eqref{ass:1}. 
Then{\rm :}

\begin{enumerate}
\item[ (ii-a)]
$\dot B^s_{p,q} (A_V)$ is independent 
of the choice of $\{ \phi_j \}_{j \in \mathbb Z}$ 
satisfying \eqref{917-1} and \eqref{917-2},  
and enjoys the following{\rm :} 
\begin{equation} \label{902-2}
\mathcal Z_V(\Omega) 
\hookrightarrow \dot B^s_{p,q } (A_V)
\hookrightarrow \mathcal Z'_V(\Omega) . 
\end{equation}

\item[ (ii-b)]
$\dot B^s_{p,q} (A_V)$ is a Banach space.
\end{enumerate}

\end{enumerate}
\end{thm}

\vskip3mm

Let us give a remark on the theorem. 
It is meaningful to consider the space $\mathcal X_V ' (\Omega)$ 
($\mathcal Z_V '(\Omega) $ resp.),
when one defines the spaces $B^s_{p,q} (A_V)$
($\dot B^s_{p,q} (A_V) $ resp.).  
In fact, when $\Omega = \mathbb R^n$ and $ V=0$,  
we see from \eqref{EQ:INC} that
$$
C_0 ^\infty (\mathbb R^n) 
\subset \mathcal S (\mathbb R^n) 
\hookrightarrow \mathcal X _0 (\mathbb R^n) 
\hookrightarrow \mathcal X' _0 (\mathbb R^n) 
\hookrightarrow \mathcal S' (\mathbb R^n). 
$$
Since $C_0 ^\infty (\mathbb R^n)$ is dense in the classical Besov spaces $B^s_{p,q}$ 
for $s \in \mathbb R$ and 
$1 \leq p , q < \infty$, 
$B^s_{p,q} $ as subspaces of $\mathcal X'_0(\mathbb R^n)$ 
are isomorphic to those as subspaces of $\mathcal S'(\mathbb R^n)$. 
Similarly, 
$\dot B^s_{p,q} $ as subspaces of $\mathcal Z'_0(\mathbb R^n)$ 
are isomorphic to those as subspaces of $\mathcal Z'(\mathbb R^n)$. 

\section{Dual spaces, embedding relations, lifting properties and isomorphic properties}

In this section, we shall introduce important properties of Besov spaces.
Let us consider the dual spaces 
of Besov spaces, lifting properties and embedding relations. 
\\

The following proposition is concerned with 
the dual spaces. 


\begin{prop}\label{thm:2}
Assume that $V$ satisfies the same assumptions as in Theorem \ref{thm:1}. 
Let $s \in \mathbb R$, 
$1 \leq p,q < \infty$, $1/p + 1/p' = 1$ and $1/q + 1/q' = 1$. 
Then the dual spaces of $B^s_{p,q}(A_V)$ and $\dot B^s_{p,q}(A_V)$ 
are $B^{-s}_{p',q'} (A_V)$ and $\dot B^{-s}_{p',q'} (A_V) $, respectively. 
\end{prop}

We have the lifting properties and embedding 
relations of our Besov spaces.

\begin{prop}\label{thm:3}
Assume that $V$ satisfies the same assumptions as in Theorem \ref{thm:1}. 
Let $\lambda_0$ be the constant as in \eqref{EQ:A-bound}, i.e., 
$A_V \geq - \lambda_0^2I$.
Let $s, s_0 \in \mathbb R$ and $1 \leq p,q, q_0 , r \leq \infty$. Then 
the following assertions hold{\rm :}
\begin{enumerate}
\item[(i)] 
The inhomogeneous Besov spaces 
enjoy the following properties{\rm :} 
\begin{equation}\notag 
\begin{split}
& 
\big\{ (\lambda _0^2 + 1) I + A_V \big\}^{s_0/2}  f 
 \in B^{s-s_0} _{p,q} (A_V) 
 \quad \text{for any } f \in B^s_{p,q} (A_V) {\rm ;} 
\\ 
& 
B^{s+\varepsilon}_{p,q} (A_V)
\hookrightarrow B^s_{p,q_0} (A_V)
\quad 
\text{for any } \varepsilon > 0 {\rm ;}
\\
& 
B^s_{p,q} (A_V)
\hookrightarrow B^{s_0}_{p, q} (A_V)
\quad \text{if } s \geq s_0 {\rm ;}
\\ 
& 
B^{s+ n (\frac{1}{r} - \frac{1}{p})}_{r,q} (A_V) 
 \hookrightarrow 
 B^s_{p,q_0} (A_V)
\quad \text{if 
$1 \leq r \leq p \leq \infty$ and 
$q \leq q_0$.} 
\end{split}
\end{equation}

\item[(ii)]
The homogeneous Besov spaces enjoy the following
properties{\rm :}
\begin{equation}\notag 
\begin{split}
& 
A_V^{s_0/2} f \in \dot B^{s-s_0} _{p,q} (A_V)
\quad \text{for any } f \in \dot B^s_{p,q} (A_V) {\rm ;}
\\
& 
\dot B^{s+ n (\frac{1}{r} - \frac{1}{p})}_{r,q} (A_V)  
 \hookrightarrow 
 \dot B^s_{p,q_0} (A_V)
\quad \text{if} \quad 
1 \leq r \leq p \leq \infty \text{ and } q \leq q_0. 
\end{split}
\end{equation}
\end{enumerate}
\end{prop}

The Besov and Lebesgue spaces
have the inclusion relation with each other.

\begin{prop}\label{thm:4}
Assume that $V$ satisfies the same assumptions as in Theorem \ref{thm:1}. Then 
the following continuous embeddings hold{\rm :} 
\begin{enumerate}
\item[(i)]
$L^p (\Omega) 
 \hookrightarrow 
 B^0_{p,2} (A_V), \dot B^0_{p,2} (A_V) $ \,\,\,
 if \,\,\,$1 < p \leq 2$. 

\item[(ii)]
$B^0_{p,2} (A_V), \dot B^0_{p,2} (A_V) 
 \hookrightarrow L^p (\Omega)$ \,\,\,
 if \,\,\,$ 2 \leq p < \infty$. 
\end{enumerate}
\end{prop}

As was stated in \S 1, the classical 
homogeneous Besov spaces are considered as subspaces 
of quotient space $\mathcal S' (\mathbb R^n) / \mathcal P (\mathbb R^n) $. 
The following proposition states that the homogeneous Besov spaces with some indices $s,p,q$ 
are characterized by subspaces of $\mathcal X'_V(\Omega)$ which is not a quotient space. 
Such characterization is known in the case of $\Omega = \mathbb R^n$ 
(see, e.g. \cite{KoYa-1994}). 
\begin{prop}\label{thm:7}
Assume that $V$ satisfies \eqref{1104-1} and \eqref{ass:1}. 
Let $s \in \mathbb R$ and $1 \leq p,q \leq \infty$.
If either $ s < n/p$ or $(s,q) = (n/p , 1)$, 
then the homogeneous Besov spaces $\dot B^s_{p,q} (A_V)$ are regarded as 
subspaces of $\mathcal X'_V (\Omega)$ according to the following isomorphism{\rm :} 
\begin{equation}\notag 
\dot B^s_{p,q} (A_V)
\simeq \Big\{ f \in \mathcal X'_V(\Omega) 
     \, \Big| \, 
     \| f \|_{\dot B^s_{p,q}(A_V)} < \infty , \,
     f = \sum _{ j \in \mathbb Z} 
        \phi_j (\sqrt{A_V} \,) f 
        \text{ in } \mathcal X'_V(\Omega)
   \Big\} . 
\end{equation}
\end{prop}

We conclude this section by stating a 
result on the equivalence relation among 
the Besov spaces by $A_0$ and $A_V$-regularity 
with $V \in L^{\frac{n}{2},\infty} (\Omega)$. 
For the definition of the Lorentz space 
$L^{\frac{n}{2},\infty} (\Omega)$, see 
\S \ref{sec:thm:6}.

\begin{prop}\label{thm:6} 
Let $ n, s,p,q$ be such that 
$$
n\geq 2, \quad 
1 \leq p,q \leq \infty, \quad 
- \min\Big\{ 2, n\Big( 1- \frac{1}{p} \Big) \Big\} 
< s < 
\min\Big\{ \frac{n}{p}, 2 \Big\} . 
$$
In addition to the same assumption on $V$ as in Theorem~\ref{thm:1}, 
we further assume that 
\begin{equation}\label{ass:1_0}
\begin{cases} 
 V \in L^1 (\Omega)
& \quad \text{if } n = 2, 
\\
V \in L^{\frac{n}{2},\infty} (\Omega)
\,\,
& \quad \text{if } n \geq 3. 
\end{cases}
\end{equation}
Then 
\begin{gather}\notag 
B^s_{p,q} (A_V) \simeq B^s_{p,q} (A_0),
\\ \notag 
\dot B^s_{p,q} (A_V) \simeq \dot B^s_{p,q} (A_0) . 
\end{gather}
\end{prop}


Let us give some remarks on 
Proposition~\ref{thm:6}.

\begin{itemize}
\item[\rm (i)] 
Proposition~\ref{thm:6} implies not only the equivalence of norms, 
but also that of the following two approximations of the identity
\begin{gather}\notag 
f = \sum _{ j \in \mathbb Z} \phi_j (\sqrt{A_0}) f  
\quad \text{in } \mathcal Z'_0 (\Omega),
 \qquad 
f = \sum _{ j \in \mathbb Z} \phi_j (\sqrt{A_V}) f  
\quad \text{in } \mathcal Z'_V (\Omega) , 
\end{gather}
for $f$ belonging to the homogeneous Besov spaces. 
Analogous approximations in $\mathcal X'_0(\Omega)$ and $\mathcal X'_V(\Omega)$
are also equivalent for the inhomogeneous Besov spaces. 

\item[\rm (ii)] 
By considering the Lorentz spaces, 
it is possible to treat the potential $V$ like 
$$
V(x) = c |x|^{-2} , \qquad c > 0, 
$$
which, in fact, $V \in L^{\frac{n}{2} , \infty} (\Omega)$.

\item[\rm (iii)] 
If $V$ is smooth more and more, then, $s$ can be taken bigger and bigger
so that the isomorphism holds. 
For instance, this comes from the following identity:
$$
(-\Delta + V) ^2 f 
= (-\Delta)^2 f + (-\Delta) (V f) + V (-\Delta) f + V^2 f
$$
when we consider the case $s = 4$. 
In fact, the term $(-\Delta) (V f)$ requires the differentiability of $V$.

\end{itemize}


\section{Key lemmas} 

In this section we introduce 
some tools and prove fundamental properties of 
$\mathcal X_V (\Omega)$ and $\mathcal Z _V (\Omega)$, which are important in later  arguments.
Here and below, we denote by $\| \cdot \|_{L^p}$ the norm of $L^p (\Omega)$ 
and $\| \cdot \|_{L^p (\mathbb R^n)}$ the norm of $L^p (\mathbb R^n)$.
\\

We start with the functional calculus; 
$L^p $-boundedness of operators $\psi (A_V)$ and $\phi_j (\sqrt {A_V})$  
for $1 \leq p \leq \infty$. 
In the previous work~\cite{IMT-preprint} 
we have established such kind of estimates for some potential $V$ when $ n \geq 3$. 
 We improve
them by some slight modifications, 
 and obtain $L^p$-estimates under more general
conditions on our potential $V$ in all space 
dimensions (see Proposition \ref{prop:Lp-bound} 
in appendix~\ref{App:AppendixA}). \\

Based on Proposition \ref{prop:Lp-bound}, 
we have the following useful lemma. 

\begin{lem}\label{lem:calc} 
Let $1 \leq r \leq p \leq \infty$. Assume that 
the measurable potential $V $ satisfies \eqref{1104-1}. 
Then 
we have the following assertions{\rm :} 
\begin{enumerate}
\item[(i)] 
For any $\phi \in C_0^\infty (\mathbb R)$ 
and $m\in \mathbb{N}\cup \{0\}$ 
there exists a constant $C > 0$ such that 
\begin{equation}\label{orig1}
\| A_V^m\phi (A_V) f \|_{L^p} 
\leq C \| f \|_{L^r}
\end{equation}
for all $f \in L^r (\Omega)$. 

\item[(ii)]
For any $\phi \in C_0^\infty ((0,\infty))$ and 
$\alpha \in \mathbb R$ there exists 
a constant $C>0$ such that 
\begin{gather}\label{inhom}
\| A_V^{ \alpha}\phi ( 2^{-j}\sqrt{A_V} ) f \| _{L^p} 
\leq C 2^{n(\frac{1}{r} - \frac{1}{p})j 
+ 2{ \alpha}j} \| f \|_{L^r} 
\end{gather}
for all $j \in \mathbb N$ and $f \in L^r (\Omega)$. 

\item[(iii)]
Assume further that $V$ satisfies \eqref{ass:1}.  
Then for any $\phi \in C_0^\infty ((0,\infty))$ and $\alpha \in \mathbb R$ there exists 
a constant $C>0$ such that 
\begin{equation}\label{homog}
\| A_V^{ \alpha} \phi ( 2^{-j}\sqrt{A_V} ) f \| _{L^p} 
\leq C 2^{n(\frac{1}{r} - \frac{1}{p})j 
+ 2{ \alpha}j} \| f \|_{L^r} 
\end{equation}
for all $j \in \mathbb Z$ and $f \in L^r (\Omega)$. 
\end{enumerate}

\end{lem}

\begin{proof}
[\bf Proof]  
Let $m \in \mathbb{N}\cup \{0\}$ and $\alpha \in \mathbb R$. 
To begin with, 
we note that the following inequality 
\begin{equation}\label{0208-1}
\|A^m_V \phi(A_V)g\|_{L^p}
\le C\|g\|_{L^p} 
\end{equation}
holds for any $g \in L^p (\Omega)$. 
In fact, writing 
$$
A_V^m\phi (A_V) = \{A_V^m e^{A_V} \phi (A_V)\}e^{-A_V} ,
$$
and noting 
$$
\lambda ^m e^{t\lambda} \phi (\lambda) \in C_0^\infty (\mathbb R), 
$$
 we conclude from Proposition~\ref{prop:Lp-bound} that 
\eqref{0208-1} holds. 
In a similar way, we get 
\begin{equation}\label{0208-2}
\|A^\alpha_V \phi( 2^{-j} \sqrt{A_V})g\|_{L^p}
\le C 2^{2 \alpha j} 
  \|g\|_{L^p} 
\end{equation}
for any $j \in \mathbb N$ and $g \in L^p (\Omega)$, 
provided that $\phi \in C_0^\infty ((0,\infty))$.

Taking account of these considerations,
we show \eqref{orig1}. 
Let $G_t (x)$ be the function of Gaussian type 
appearing in the pointwise estimate \eqref{1117-4_2} of  kernel of $e^{-tA_V}$
from Lemma \ref{lem:pointwise},  i.e.,
\[
G_t(x)=C t^{-\frac{n}{2}}\exp\left(
-\frac{|x|^2}{Ct}
\right), 
\qquad t > 0 , \quad x \in \mathbb R^n, 
\]
where $C$ is a certain positive constant.
We write 
$$
 \phi (2^{-j}\sqrt{A_V}) f 
 = e^{- 2^{-2j} A_V } \big\{ e^{2^{-2j}A_V}  \phi (2^{-j}\sqrt{A_V}) \big\} f  . 
$$
By using pointwise estimate \eqref{1117-4_2} for $e^{-tA_V}$, we have 
\begin{equation}\label{G-1}
\begin{split}
\big| e^{- 2^{-2j}A_V}  f (x) \big|
\leq 
& 
 \int_{\mathbb R^n}
   G_{2^{-2j}} (x-y) \big| \tilde f(y) \big| \, dy ,
   \quad j \in \mathbb N, \quad   x \in \Omega, 
\end{split}
\end{equation}
where $\tilde f$ is a zero extension of $f$ outside of $\Omega$. 
Let $r_0$ be such that $1/p = 1/r_0 + 1/r-1$. 
Then we conclude from the estimates \eqref{0208-1}, 
\eqref{G-1}
 and Young's inequality that 
\begin{equation}\notag 
\begin{split}
\| A_V^m \phi (A_V) f \|_{L^p}
& = \big\| \big\{ A_V ^m e^{A_V} \phi (A_V) \big\} e^{-A_V} f \big\|_{L^p}
\\
& 
\leq C \| e^{-A_V} f \|_{L^p} 
\\
& 
\leq C \| G_1 * |\tilde f| \|_{L^p (\mathbb R^n)}
\\
& 
\leq C \| G_1 \|_{L^{r_0 } (\mathbb R^n)} \| \tilde f \|_{L^r (\mathbb R^n)}
\\
& \leq C \| f \|_{L^r} . 
\end{split}
\end{equation}
This proves \eqref{orig1}.
 
As to \eqref{inhom}, again by using \eqref{0208-2}, \eqref{G-1} and Young's inequality, 
we get
\begin{equation}\notag 
\begin{split}
\| A_V^{\alpha} \phi (2^{-j} \sqrt{A_V}) f \|_{L^p}
& = 2^{2{ \alpha}j} 
 \big\| \big\{ (2^{-2j}A_V)^{ \alpha} e^{2^{-2j}A_V} \phi (2^{-j} \sqrt{A_V}) \big\} e^{-2^{-2j}A_V} f \big\|_{L^p}
\\
& 
\leq C 2^{2\alpha j} \| e^{-2^{-2j}A_V} f \|_{L^p} 
\\
& 
\leq C 2^{2 \alpha j} \| G_{2^{-2j}} * |\tilde f| \|_{L^p (\mathbb R^n)}
\\
& 
\leq C 2^{2 \alpha j}\| G_{2^{-2j}} \|_{L^{r_0 } (\mathbb R^n)} \| \tilde f \|_{L^r (\mathbb R^n)}
\\
& \leq C 2^{2 \alpha j} 2^{n(\frac{1}{r}-\frac{1}{p})j}\| f \|_{L^r}
\end{split}
\end{equation}
for any $j \in \mathbb N$, which proves \eqref{inhom}. 
The estimate \eqref{homog} is also proved in the analogous way to the above argument
by applying 
\eqref{bd:hom} in Proposition~\ref{prop:Lp-bound} and 
\eqref{1117-5} in Lemma~\ref{lem:pointwise} instead of 
\eqref{bd:inh} in Proposition~\ref{prop:Lp-bound} and 
\eqref{1117-4_2} in Lemma~\ref{lem:pointwise},
respectively. 
The proof of Lemma~\ref{lem:calc} is finished.
\end{proof}

The second lemma concerns with the completeness of test function spaces. 

\begin{lem}\label{EQ:Frechet}
Assume that the measurable potential $V$
satisfies \eqref{1104-1}. Then 
$\mathcal X_V (\Omega)$ is complete. 
In addition to the assumption \eqref{1104-1}, if 
$V$ satisfies \eqref{ass:1}, then 
$\mathcal Z_V (\Omega)$ is complete.
\end{lem}
\begin{proof}[\bf Proof]
We first show the completeness of $\mathcal X_V (\Omega)$. 
Let $\{ f_N \}_{N=1} ^\infty$ be a Cauchy sequence in $\mathcal X_V(\Omega)$. 
Then, for $M = 1,2, \dots$, there exists $C_M > 0$ such that 
\begin{equation}\label{1105-1}
p_{V,M} (f_N) \leq  C_M  
\quad \text{for all } N \in \mathbb N. 
\end{equation}
Since $\{ f_N \}$ is a Cauchy sequence in $ L^1 (\Omega)$, 
there exists a function $f \in  L^1 (\Omega)$ such that 
$$
f_N \to f \quad \text{ in }  L^1 (\Omega)
\text{ as } N \to \infty . 
$$
Combining this convergence with  
the 
boundedness
of 
$2^{Mj}\phi_j (\sqrt{A_V})$ from $ L^1 (\Omega)$ to itself, 
which is assured by \eqref{inhom} 
for $\alpha = 0$ and \eqref{1105-1}, 
we have 
$$
2^{Mj} \| \phi_j (\sqrt{A_V}) f \|_{ L^1 }  
= 
\lim _{N \to \infty} 
2^{Mj} \| \phi_j (\sqrt{A_V}) f_N \|_{ L^1 },  
$$
and hence,
$$
p_{V,M} (f) \leq  C_M 
$$
for $M = 1,2,\dots$. Hence we get 
$f \in \mathcal X_V (\Omega)$. 
We next show the convergence of $f_N$ to $f$ in $\mathcal X_V (\Omega)$. 
For each $M$, let us take a subsequence 
$\{ f_{N(k)} \}_{k = 1}^\infty$ such that 
$$
p_{V,M} (f_{N( k)}  -f _{N( k-1 )}) \leq 2^{-k},
$$ 
where we put $f_{N(0)}=0$. Hence we have 
\begin{equation}\label{1105-2}
\sum _{k=1} ^\infty
p_{V,M} (f_{N(k)}  -f _{N(k-1)})
< \infty. 
\end{equation}
Since 
$\{ f_{N(k)} \}_{k = 1}^\infty$ is a Cauchy sequence in 
$ L^1 (\Omega)$, 
$f$ is written by 
\begin{equation}\label{1105-3}
f = \lim _{L \to \infty} f_{N(L)} 
= \lim _{L \to \infty}\sum _{k = 1} ^L (f_{N(k)}  -f _{N(k-1)}) 
\quad \text{in } L^1 (\Omega). 
\end{equation}
Then \eqref{1105-2} and \eqref{1105-3} yield the convergence of 
$p_{V,M} (f_{N(L)} - f)$ to zero as 
$L \to \infty$, and hence, 
$$
p_{V,M} (f_{N} - f) \to 0 
\quad \text{as } N \to \infty 
\quad \text{for } M = 1,2,\cdots .
$$
Therefore, $\mathcal X_V(\Omega)$ is complete. 

We next show the completeness of $\mathcal Z_V (\Omega)$.
Let $\{ f_N \}_{N=1} ^\infty$ be a Cauchy sequence in $\mathcal Z_V(\Omega)$. 
Since $\mathcal Z_V (\Omega)$ is a subspace of $\mathcal X_V(\Omega)$ 
and $\mathcal X_V (\Omega)$ is complete, 
$\{ f_N \}_{N=1}^\infty$ is also a Cauchy sequence in $\mathcal X_V(\Omega)$  
and there exists an element 
$f \in \mathcal X_V(\Omega)$ such that 
$f_N $ converges to $f$ in $\mathcal X_V(\Omega)$ as $N \to \infty$. 
In order to prove $f \in \mathcal Z_V (\Omega)$, 
we show that 
\begin{equation}\label{901-1}
\sup _{ j \leq 0} 2^{M |j| } \| \phi_j (\sqrt{A_V} ) f\|_{L^1 } < \infty 
\quad \text{for } M = 1,2, \cdots . 
\end{equation}
Since $f_N$ converges to $f$ in $L^1 (\Omega)$ as $N \to \infty$ 
and $\phi_j (\sqrt{A_V})$ is bounded on 
$L^1 (\Omega)$ for each $j \in \mathbb Z$ 
by \eqref{homog} for $\alpha = 0$, 
it follows that 
$$
\lim _{N \to \infty} \| \phi_j (\sqrt{A_V} ) f_N \|_{L^1} 
= \| \phi_j (\sqrt{A_V} ) f \|_{L^1} 
\quad \text{for any } j \in \mathbb Z .
$$
Since $\{ f_N \}_{N=1}^\infty$ is a Cauchy sequence in $\mathcal Z_V(\Omega)$, 
$\{ q_{ V,M}(f_N) \}_{N=1}^\infty$ is a bounded sequence for each $M$ and 
there exists a constant $C_M>0$ 
depending only on $M$ such that 
$$
2^{M |j| } \| \phi_j (\sqrt {-\Delta}) f_N \|_{L^1} 
\leq C_M 
\quad \text{for all } j \leq 0 \text{ and } N = 1,2,\cdots . 
$$
By taking the limit as $N \to \infty $ in the above inequality,  
we conclude that $f$ satisfies \eqref{901-1}, 
and hence, $f \in \mathcal Z_V (\Omega)$. 
Finally, the convergence of $f_N$ to $f$ in $\mathcal Z_V (\Omega)$ 
follows from the analogous argument to 
\eqref{1105-2} and \eqref{1105-3}:
\begin{equation*}
\sum _{k=1} ^\infty
q_{V,M} (f_{N(k)}  -f _{N(k-1)})
< \infty , 
\end{equation*}
\begin{equation*}
f = \lim _{L \to \infty}\sum _{k = 1} ^L (f_{N(k)}  -f _{N(k-1)}) 
\quad \text{in }  L^1 (\Omega), 
\end{equation*}
which imply that
$$
q_{V,M} (f_{N} - f) \to 0 
\quad \text{as } N \to \infty 
\quad \text{for } M = 1,2,\cdots .
$$
Thus we conclude that $\mathcal{Z}_V(\Omega)$
is complete. 
The proof of Lemma \ref{EQ:Frechet} is finished. 
\end{proof}

The third lemma is useful in proving Lemma \ref{lem:decomposition1}. 

\begin{lem}\label{lem:map_M}
Assume that the measurable potential $V$
satisfies \eqref{1104-1}. Then the following assertions hold{\rm :}
\begin{enumerate}
\item[(i)] 
For any $f \in \mathcal X'_V (\Omega)$, 
there exist a number
$M_0 \in \mathbb N$ and a constant $C_f > 0$ such that 
\begin{equation}\notag 
\big| 
{}_{{\mathcal X}_V' }
\langle f , g \rangle _{{\mathcal X}_V} 
\big| 
\leq C_f\, p_{V,M_0} (g) 
\quad \text{for any } g \in \mathcal X_V (\Omega) . 
\end{equation}
\item[(ii)] 
In addition to the assumption \eqref{1104-1}, if 
$V$ satisfies \eqref{ass:1}, then 
for any $f \in \mathcal Z'_V (\Omega)$, 
there exist a number
$ M_1 \in \mathbb N$ and a constant $C_f > 0$ such that 
\begin{equation}\notag 
\big| 
{}_{{\mathcal Z}_V' }
\langle f , g \rangle _{{\mathcal Z}_V} 
\big| 
\leq C_f q_{V, M_1} (g) 
\quad \text{for any } g \in \mathcal Z_V (\Omega) . 
\end{equation}
\end{enumerate}
\end{lem}
\begin{proof}[\bf Proof.]
Suppose that {\rm (i)} is not true. 
Then  
for any $m \in \mathbb N$ there exists $g_m \in \mathcal X_V (\Omega )$ 
such that 
\begin{equation}\label{1112-1}
\big| 
{}_{{\mathcal X}_V' }
\langle f , g_m \rangle _{{\mathcal X}_V} 
\big| 
> m p_{V,m} (g_m).
\end{equation}
Put 
$$
\widetilde g_m := 
\frac{g_m}{ m p_{V,m} (g_m)} . 
$$
Noting that $p_{V,k} (\widetilde g_m)$ is 
monotonically increasing in $k\in \{1,2,\dots,m\}$,
we have
$$
p_{V,k} (\widetilde g_m) 
\leq p_{V,m} (\widetilde g_m) 
= \frac{1}{m} 
\quad \text{ for } k = 1,2,\cdots , m.
$$
Hence it follows that 
for any fixed $k \in \mathbb N$
$$
p_{V,k} (\widetilde g_{m}) \to 0
\quad \text{as } m \to \infty;
$$
thus we find that 
$$
 \widetilde g_m \to 0 \quad 
 \text{in } \mathcal X_V (\Omega) 
 \text{ as } m \to \infty . 
$$
The above convergence yields that
\begin{equation}\label{EQ:contra}
\big| 
{}_{{\mathcal X}_V' }
\langle f , \widetilde g_m \rangle _{{\mathcal X}_V} 
\big| 
\to 0 
\quad \text{as } m \to \infty. 
\end{equation}
However, the assumption \eqref{1112-1} implies that 
$$
\big| 
{}_{{\mathcal X}_V' }
\langle f , \widetilde g_m \rangle _{{\mathcal X}_V} 
\big| > 1 \quad \text{for all $m\in \mathbb{N}$;}
$$
therefore this inequality contradicts 
\eqref{EQ:contra}. 
Thus the assertion {\rm (i)} holds. 
The assertion {\rm (ii)} follows analogously. 
This ends the proof of Lemma \ref{lem:map_M}.
\end{proof}

The following lemma states that the mapping $\phi (A_V)$ 
is well-defined on $\mathcal X_V (\Omega)$, $\mathcal Z_V (\Omega) $ 
and their duals. 

\begin{lem}\label{lem:mapping}
Assume that the measurable potential $V$
satisfies \eqref{1104-1}. Then 
the following assertions hold{\rm :}

\begin{enumerate}
\item[(i)] 
For any $\phi \in C_0^\infty (\mathbb R)$, 
$\phi (A_V)$ maps continuously from 
$\mathcal X_V (\Omega)$
into itself, and maps continuously 
from $\mathcal X'_V (\Omega)$ into itself.

\item[(ii)] 
In addition to the assumption \eqref{1104-1}, 
if $V$ satisfies \eqref{ass:1}, then 
for any $\phi \in C_0^\infty ((0,\infty))$, 
$\phi (A_V)$ maps continuously from 
$\mathcal Z_V (\Omega)$ into itself, and 
maps continuously from $\mathcal Z'_V (\Omega)$ 
into itself. 
\end{enumerate}
\end{lem}
\begin{proof}[\bf Proof.]
First we prove the assertion {\rm (i)}. 
Let $f \in \mathcal X_V (\Omega)$. 
It follows from \eqref{orig1} 
in Lemma \ref{lem:calc} 
that 
\begin{equation}\label{1105-6}
A_V^m \phi(A_V)f\in \mathcal D(A_V),
 \quad 
p_{V,M} ( \phi(A_V)f)  
\leq C p_{V,M} ( f)  
\end{equation}
for $m= 0,1,2,\ldots$; 
$M = 1,2,\dots$.
This proves that $\phi (A_V)$ is 
continuous from $\mathcal X_V (\Omega)$ into
itself.  
The continuity of $\phi (A_V)$ from 
$\mathcal X'_V (\Omega)$ into itself 
follows from the definition \eqref{901-4}.

As to the assertion (ii), 
since $V$ satisfies \eqref{1104-1}, 
$\phi (A_V)$ enjoys the assertion {\rm (i)}, 
and hence,
we conclude that 
$$
\phi (A_V) f \in \mathcal X_V (\Omega) 
\quad \text{for any } f \in \mathcal Z_V (\Omega). 
$$
We show that
\begin{equation}\label{EQ:q-esti}
q_{V,M} (\phi(A_V)f)\le Cq_{V,M}(f)
\end{equation}
for $M=1,2,\ldots$. 
Indeed, recalling the definition 
\eqref{EQ:qV} of $q _{V,M} (f)$ and noting that 
$$
q_{V,M}(\phi(A_V)f) \leq 
p_{V,M}(\phi(A_V)f)
+ \sup_{j \leq 0} 2^{M |j|} \| \phi_j (\sqrt{A_V}) \phi(A_V)f \|_{L^1},
$$
we apply \eqref{1105-6} to the first term to obtain
$$
p_{V,M} (\phi (A_V)f) 
\leq C p _{V,M} (f) 
\leq C q _{V,M} (f). 
$$
For the second term in $q_{V,M} ( \phi (A_V)f)$, 
again applying \eqref{orig1} for $m = 0$, 
we estimate  
\begin{equation}\notag 
\begin{split}
 \sup _{ j \leq 0} 2^{M |j|} 
 \| \phi_j (\sqrt{A_V} ) \phi (A_V)f \|_{L^1} 
\leq 
& 
C 
 \sup _{ j \leq 0} 2^{M |j|} 
 \| \phi_j (\sqrt{A_V} ) f \|_{L^1}  
\\
 \leq 
 & 
  C q_{V,M} (f)
\end{split}
\end{equation}
for $M = 1,2, \dots$. Therefore, the 
above two estimates imply \eqref{EQ:q-esti},
which concludes the continuity of 
$\phi (A_V)$ from $\mathcal Z_V (\Omega)$ 
into itself. Finally, 
the continuity of $\phi (A_V)$ from $\mathcal Z'_V (\Omega)$ into itself 
follows from the definition \eqref{901-5}. 
The proof of Lemma \ref{lem:mapping} is finished. 
\end{proof}

The approximation of identity is established by the following lemma. 

\begin{lem}\label{lem:decomposition1}
Assume that the measurable potential $V$
satisfies \eqref{1104-1}. Then 
the following assertions hold{\rm :}
\begin{enumerate}
\item[(i)] 
For any $f \in \mathcal X_V(\Omega)$, we have
\begin{equation}\label{907-1}
f = \psi (A_V) f 
    + \sum _{ j \in \mathbb N} \phi_j (\sqrt{A_V}) f 
   \quad \text{in} \quad 
    \mathcal X_V (\Omega). 
\end{equation}
Furthermore, for any $f \in \mathcal X_V' (\Omega)$, 
we have also  the identity \eqref{907-1} 
in $\mathcal X '_V (\Omega)$, 
and $\psi (A_V) f$ and $\phi_j (\sqrt{A_V}) f$ are regarded as elements 
in $L^\infty (\Omega)$. 

\item[(ii)] In addition to the assumption \eqref{1104-1}, 
if $V$ satisfies \eqref{ass:1}, then 
for any $f \in \mathcal Z_V(\Omega)$, we have 
\begin{equation}\label{907-2}
f =  \sum _{ j \in \mathbb Z} \phi_j (\sqrt{A_V}) f 
   \quad \text{in} \quad 
    \mathcal Z_V (\Omega). 
\end{equation}
Furthermore, for $f \in \mathcal Z_V' (\Omega)$, we have also the identity \eqref{907-2} 
in $\mathcal Z '_V (\Omega)$, 
and $\phi_j (\sqrt{A_V}) f$ are regarded as  elements in $L^\infty (\Omega)$.  
\end{enumerate}
\end{lem}


\begin{proof}
[\bf Proof] 
First we prove the assertion {\rm (i)}. 
Let $f \in \mathcal X_V (\Omega)$. 
Then we have $f \in L^2 (\Omega)$, and $f$ is written as 
$$
f = \psi (A_V) f 
   + \sum _{ j \in \mathbb N} \phi_j (\sqrt{A_V}) f 
   \quad \text{in} \quad 
   L^2 (\Omega).
$$
It is sufficient to verify that the series in 
the right member is absolutely convergent in 
$\mathcal X_V(\Omega)$. 
Let $M \in \mathbb N$ be arbitrarily fixed. 
Applying \eqref{inhom} for $\alpha = 0, 1$ from Lemma \ref{lem:calc}, we have 
\[
p_{V,M} \big(\psi (A_V) f \big) 
\leq C p_{V,M} (f),
\]
\begin{equation}\notag 
\begin{split}
p_{V,M} \big( \phi_j (\sqrt{A_V}) f \big) 
\leq
& C 2^{-2j} p_{V,M} \big( A_V \phi_j (\sqrt{A_V})  f \big)
\\
\leq 
&
 C 2^{-2j} p_{V,M+2} (f),
\end{split}
\end{equation}
which imply that 
\begin{equation}\label{1105-8}
\sum _{ j \in \mathbb N} 
p_{V,M} \big(\phi_j (\sqrt{A_V}) f \big) 
\leq C p_{V,M+2} (f) \sum _{j \in \mathbb N} 2^{-2j}  < \infty.
\end{equation}
Hence 
\eqref{907-1} holds for $f\in \mathcal{X}_V(\Omega)$.
As to the expansion \eqref{907-1} 
for $f \in \mathcal X_V '(\Omega)$, 
applying the identity \eqref{907-1} for $g\in\mathcal{X}_V(\Omega)$, we have 
formally the following identity: 
\begin{equation}\label{1112-2}
\begin{split}
{}_{\mathcal X'_V} \langle f , g \rangle_{\mathcal X_V} 
& 
= 
{}_{\mathcal X'_V} \langle f , \psi (A_V) g \rangle_{\mathcal X_V} 
+ 
\sum _{ j \in \mathbb N} 
{}_{\mathcal X'_V} \langle f , \phi_j (\sqrt{A_V}) g \rangle_{\mathcal X_V} 
\\
& 
= 
{}_{\mathcal X'_V} \langle \psi (A_V) f , g \rangle_{\mathcal X_V} 
+ 
\sum _{ j \in \mathbb N} 
{}_{\mathcal X'_V} \langle \phi_j (\sqrt{A_V}) f , g \rangle_{\mathcal X_V} ,
\end{split}
\end{equation}
where the second equality is valid due to the definition \eqref{901-4}. 
We must prove the absolute convergence of the series in \eqref{1112-2}.
By Lemma \ref{lem:map_M} {\rm (i)}, 
there exist $M_0 \in \mathbb N$ and $C > 0$ such that 
\begin{equation}\notag 
\begin{split}
\big|
{}_{\mathcal X'_V} \langle \phi_j (\sqrt{A_V}) f , g \rangle_{\mathcal X_V}
\big| 
=
& 
 \big|
{}_{\mathcal X'_V} \langle f , \phi_j (\sqrt{A_V}) g \rangle_{\mathcal X_V}
\big| 
\\
\leq 
& 
 C_f p_{V,M_0}(\phi_j (\sqrt{A_V}) g) . 
\end{split}
\end{equation}
Then, the above estimate and \eqref{1105-8} yield 
the absolute convergence of the series in \eqref{1112-2}.

For the proof of $\psi (A_V) f 
\in L^\infty (\Omega)$, 
we begin by proving that 
\begin{equation}\label{1117-1}
\big|
{}_{\mathcal X'_V} \langle \psi (A_V) f , g \rangle_{\mathcal X_V}
\big| 
\leq C \| g \|_{L^1} 
\quad \text{for all } g \in \mathcal X_V (\Omega). 
\end{equation}
By the definition \eqref{901-4}, Lemma \ref{lem:map_M} {\rm (i)} and 
\eqref{orig1} for $m = 0$, 
there exist $M_0 \in \mathbb N$ 
and $C_f, C_{f,\psi} > 0$ such that 
\begin{equation}\notag 
\begin{split}
\big|
{}_{\mathcal X'_V} \langle \psi (A_V) f , g \rangle_{\mathcal X_V}
\big| 
= 
& 
\big|
{}_{\mathcal X'_V} \langle f , \psi (A_V) g \rangle_{\mathcal X_V}
\big| 
\\
\leq 
& 
C_f p_{V,M_0} (\psi (A_V) g) 
\\
 \leq 
 &
 C_{f,\psi}  \| g \| _{L^1},
\end{split}
\end{equation}
which proves \eqref{1117-1}. 
Thanks to \eqref{1117-1}, 
the Hahn-Banach theorem allows us to 
deduce that the mapping 
$$
{}_{\mathcal X'_V} \langle \psi (A_V) f , \cdot \rangle_{\mathcal X_V}
: \mathcal X_V (\Omega) \to \mathbb C
$$ 
is extended as a mapping 
from $L^1 (\Omega)$ to $\mathbb C$. 
Since $L^1 (\Omega)^* = L^\infty (\Omega)$, there exists 
a function $F \in L^\infty(\Omega)$ 
such that 
$$
{}_{\mathcal X'_V} \langle \psi (A_V) f , g \rangle_{\mathcal X_V}
= \int_{\Omega} F (x) \overline{g (x)} \, dx  
\quad 
\text{for all } g \in \mathcal X_V (\Omega) . 
$$
Then we conclude that $\psi (A_V) f \in L^\infty (\Omega)$. In a similar way, 
it is possible to prove that 
$\phi_j (\sqrt{A_V})f \in L^\infty (\Omega)$. 
The proof of {\rm (i)} is now complete.

As to the assertion {\rm (ii)}, 
noting that 
any $f\in \mathcal{Z}_V(\Omega)$ is in $L^2 (\Omega)$, 
we first prove that 
\begin{equation}\label{EQ:L2}
f = \sum _{j \in \mathbb Z} \phi_j (\sqrt{A_V}) f
\quad \text{in } L^2 (\Omega) 
\end{equation}
for any $f \in L^2 (\Omega)$. 
Put 
\begin{equation}\label{EQ:g-n}
g _L := 
\int_{-\infty} ^{ \infty } 
     \Big( 1 - \sum _{j \geq L} \phi_j (\sqrt{\lambda}) \Big) 
  d E_{A_V} (\lambda) f. 
\end{equation}
It is readily checked that $\{ g_L \}$ is a Cauchy sequence in $L^2 (\Omega)$, 
so we put 
$$
  g := \lim _{L \to - \infty}  g_L
\quad 
\text{ in } L^2 (\Omega).
$$
Noting that $A_V$ is non-negative on $L^2 (\Omega)$ 
and that the support of $1 - \sum _{j \geq L} \phi_j (\sqrt{\lambda})$ is contained in the interval 
$(-\infty, 2^{2L}]$, we find that 
\begin{align*}
\|A_Vg_L\|^2_{L^2}=&
 \int_{-\infty} ^{2^{2L}} 
  \Big|\lambda \Big( 1 - \sum _{j \geq L} \phi_j (\sqrt{\lambda}) \Big)\Big|^2 
  d \|E_{A_V} (\lambda) f \|^2_{L^2}\\
\leq& C 2^{4L} \|f\|^2_{L^2} 
\to 0 
\quad \text{as } L \to - \infty. 
\end{align*}
Hence we deduce that
\[
g \in \mathcal D (A_V), \qquad A_V g = 0 \quad \text{in $L^2(\Omega)$}
\]
by the fact that $g_ L \in \mathcal D (A_V)$, the definition of $g$, 
and the closeness of $A_V$ on $L^2(\Omega)$. 
Since zero is not an eigenvalue of $A_V$ by Lemma~\ref{lem:zero}, 
we conclude 
that 
$g=0$, which proves \eqref{EQ:L2} for any $f \in L^2 (\Omega)$.

Now, as in the previous argument, it is 
sufficient to show that the series in 
the right member of \eqref{EQ:L2} 
is absolutely convergent in 
$\mathcal Z_V(\Omega)$.  
For the series \eqref{EQ:L2} with $ j \geq 1$, the absolute convergence is obtained by 
the same argument as \eqref{1105-8}. 
For the case $j \leq 0$, 
it follows from \eqref{homog} for 
$\alpha = \pm 1$ that 
\begin{equation}\notag 
\begin{split}
q_{V,M} \big( \phi_j (\sqrt{A_V}) f \big) 
\leq C 2^{2j} q_{V,M} \big( A_V ^{-1}\phi_j (\sqrt{A_V}) f \big) 
\leq C 2^{2j} q_{V,M+2} ( f ), 
\end{split}
\end{equation}
which imply that 
\begin{equation}\notag 
\sum _{ j \leq 0} q_{V,M} \big( \phi_j (\sqrt{A_V}) f \big) 
\leq C q_{V,M+2} (f) \sum _{ j \leq 0 } 2^{2j} 
< \infty 
\end{equation}
for all $M \in \mathbb N$. 
Therefore, \eqref{907-2} is verified {for $f\in \mathcal{Z}_V(\Omega)$}.

Finally, as to 
the identity \eqref{907-2} 
for $f \in \mathcal Z'_V(\Omega)$, 
we proceed the analogous argument to 
that with replacing the assertion {\rm (i)} 
for $p_{V,M}$ and Lemma \ref{lem:map_M} (i)
by $q_{V,M}$ and Lemma 
\ref{lem:map_M} (ii), respectively. 
The proof of $\phi_j (\sqrt{A_V}) f \in L^\infty (\Omega)$ also follows 
from the analogous argument to 
that of 
the assertion {\rm (i)} as above. So we 
may omit the details. 
The proof of Lemma \ref{lem:decomposition1} is complete.  
\end{proof}

As a consequence of  
Lemmas \ref{lem:calc} and \ref{lem:decomposition1}, we have: 
\begin{lem}\label{cor:1}
The following inclusion relations hold{\rm :}
\begin{gather}
\label{0204-1}
\mathcal X_V (\Omega)  
\subset L^1 (\Omega) \cap L^\infty (\Omega) ,
\\
\label{0204-2}
L^p (\Omega) 
\subset \mathcal X_V ' (\Omega) 
\quad \text{for any }  1 \leq p \leq \infty . 
\end{gather}
As a consequence, we have 
\begin{gather}
\label{0204-3}
\mathcal Z_V (\Omega)  
\subset L^1 (\Omega) \cap L^\infty (\Omega) ,
\\
\label{0204-4}
L^p (\Omega) 
\subset \mathcal Z_V ' (\Omega) 
\quad \text{for any }  1 \leq p \leq \infty . 
\end{gather}

\end{lem}

\begin{proof}
[\bf Proof] 
Once \eqref{0204-1} and \eqref{0204-2} are proved, 
\eqref{0204-3} and \eqref{0204-4} hold, since 
$$
\mathcal Z_V (\Omega) 
\subset \mathcal X_V (\Omega) 
\quad \text{and} \quad 
 \mathcal X_V' (\Omega)
\subset \mathcal Z_V' (\Omega) . 
$$

We show the inclusion relation \eqref{0204-1}. 
Put  
$$
\Phi_j := \phi_{j-1} + \phi_j + \phi_{j+1}. 
$$
Let $f \in \mathcal X_V (\Omega)$. Then it follows from 
the definition of semi-norms $p_{V,M} (\cdot)$ that 
$$
\| f \|_{L^1 } \leq p_{V,0} (f) .
$$
As to the $L^\infty$-norm, we deduce from 
the identities \eqref{907-1}, 
$\phi_j = \Phi_j \phi_j$
and the estimate \eqref{inhom} for $\alpha = 0$ that 
\begin{gather}\notag
\begin{split}
\| f \|_{L^\infty} 
\leq 
& \| \psi (A_V) f \|_{L^\infty} 
  + \sum _{ j \in \mathbb N} 
    \|  \Phi_j (\sqrt{A_V} ) \phi_j (\sqrt{A_V} ) f \|_{L^\infty}
\\
\leq 
& C \| f \|_{L^1} 
  + C \sum _{ j \in \mathbb N} 2^{-j} \cdot 2^j
    2^{nj} \| \phi_j (\sqrt{A_V} ) f \|_{L^1}
\\
\leq 
& C p_{V,0} (f) 
  + C \sum _{j \in \mathbb N} 2^{-j} 
     \sup _{ k \in \mathbb N} 2^{(n+1) k} 
     \| \phi_k (\sqrt{A_V} ) f \|_{L^1}
\\
\leq 
& C p_{V,n+1} (f) . 
\end{split}
\end{gather}
Summarizing the above estimates now, 
we conclude the inclusion relation \eqref{0204-1}. 

Finally, we prove the inclusion relation \eqref{0204-2}. 
Let $f \in L^p (\Omega)$ and $g \in \mathcal X_V (\Omega)$.  
Then it follows from H\"older's inequality and the above two estimates that 
\begin{equation}\notag 
\begin{split}
\int_{\Omega} | f(x) g(x)| dx 
\leq 
& 
\| f \|_{L^p} \| g \|_{L^{p'}}
\\
\leq 
& 
\| f \|_{L^p} \| g \|_{L^1 \cap L^\infty} 
\\
\leq 
& 
C \| f \|_{L^p} p_{V,n+1} (g),  
\end{split}
\end{equation}
where $p'$ is the conjugate exponent of $p$. 
This estimate means that $f \in L^p (\Omega)$ belongs to 
$\mathcal X_V ' (\Omega)$. 
Hence we conclude \eqref{0204-2}. 
The proof of Lemma \ref{cor:1} is complete. 
\end{proof}

\section{Proof of Theorem \ref{thm:1}} 
In this section we prove Theorem \ref{thm:1}.


\begin{proof}[\bf Proof of independence of the choice of $\psi$ and $\{\phi_j\}$]  
 The proof of the independence in {\rm (i-a)} and {\rm (ii-a) is similar 
 to that of Triebel \cite{Triebel_1983}. }
As to {\rm (i-a)}, 
let us take $\psi = \psi ^{(k)}$, $\phi_j = \phi_j ^{(k)}$ 
($k = 1,2$) satisfying \eqref{917-1}, 
\eqref{917-2} and \eqref{917-3}.
Since $\psi ^{(1)}$ and $\phi_j ^{(1)}$ satisfy 
\begin{gather}\notag 
\psi^{(1)} = \psi ^{(1)} \big( \psi ^{(2)} + \phi_1 ^{(2)} \big), 
\quad 
\phi_1^{(1)} = \phi_1 ^{(1)} \big( \psi ^{(2)} + \phi_1 ^{(2)} + \phi_2 ^{(2)} \big), 
\\ \label{1118-1}
\quad 
\phi^{(1)} _j  = \phi^{(1)} _j \big( \phi^{(2)} _{j-1} + \phi^{(2)} _j + \phi^{(2)} _{j+1} \big) 
\quad \text{ for } j = 2,3,\cdots, 
\end{gather}
it follows from \eqref{orig1} and \eqref{inhom} 
in Lemma \ref{lem:calc} that 
\begin{gather}\notag 
\| \psi^{(1)}(A_V) f \|_{L^p} + \| \phi_1^{(1)}(\sqrt{A_V}) f \|_{L^p} 
\leq C\Big\{ \| \psi^{(2)}(A_V) f \|_{L^p} 
   + \sum _{k=1}^2\| \phi_k ^{(2)} (\sqrt{A_V}) f \|_{L^p}\Big\},
\\ \notag 
\| \phi_j ^{(1)} (\sqrt{A_V}) f \|_{L^p} 
\leq C \sum _{k = -1}^{1} \| \phi_{j+k} ^{(2)} (\sqrt{A_V}) f \|_{L^p} 
\quad \text{ for } j = 2,3,\cdots,
\end{gather}
which imply that
\begin{equation}\notag 
\begin{split}
& 
\| \psi^{(1)}(A_V) f \|_{L^p} 
+  \big\| \big\{ 2^{sj} \| \phi_j ^{(1)} (\sqrt{A_V}) f \|_{L^p} \big\}_{j \in \mathbb N} 
    \big\|_{\ell^q (\mathbb N)}
\\
\leq & 
   C\Big\{ \| \psi^{(2)}(A_V) f \|_{L^p} 
+  \big\| \big\{ 2^{sj} \| \phi_j ^{(2)} (\sqrt{A_V}) f \|_{L^p} \big\}_{j \in \mathbb N} 
    \big\|_{\ell^q (\mathbb N)}\Big\}.
\end{split}
\end{equation}
This proves the independence in {\rm (i-1)} 
for the inhomogeneous Besov spaces. 

As to {\rm (ii-a)}, 
we use the identity 
\eqref{1118-1} for all $j \in \mathbb Z$ 
and apply \eqref{homog} for $\alpha=0$ in Lemma \ref{lem:calc} to get 
\[
\big\| \big\{ 2^{sj} \| \phi_j ^{(1)} (\sqrt{A_V}) f \|_{L^p} \big\}_{j \in \mathbb Z} 
    \big\|_{\ell^q (\mathbb Z)}
\leq 
   C\Big\{ \big\| \big\{ 2^{sj} \| \phi_j ^{(2)} (\sqrt{A_V}) f \|_{L^p} \big\}_{j \in \mathbb Z} 
    \big\|_{\ell^q (\mathbb Z)}\Big\}.
\]
This ends the proof of 
the required independence of the choice of 
$\psi$ and $\{\phi_j\}$.
\end{proof}

\begin{proof}
[\bf Proof of inclusion relations \eqref{902-1} and \eqref{902-2}] 
Let $p'$ and $q'$ be such that $1/p + 1/p' = 1$ and $1/q + 1/q' = 1$. 
First we prove the embedding 
\eqref{902-1}, namely,  
$$
\mathcal X_V (\Omega) \hookrightarrow B^s_{p,q} (A_V)
\hookrightarrow \mathcal X_V ' (\Omega).
$$ 
Take $\Psi$ and $\Phi_j$ such that 
$$
\Psi := \psi + \phi_1, 
\quad 
\Phi_1 := \psi + \phi_1 + \phi_2, 
\quad 
\Phi_j := \phi_{j-1} + \phi_j + \phi_{j+1} 
\text{ for } j = 2,3,\cdots.
$$
Let $M \in \mathbb N$ be such that 
$M > s + n(1-1/p)$. 
Then, for any $f \in \mathcal X_V(\Omega)$, 
we deduce from the identities  
$\phi_j = \Phi_j \phi_j$ and the estimate \eqref{inhom} for $\alpha = 0$ 
in Lemma \ref{lem:calc} that 
\begin{equation}\notag 
\begin{split}
\| f \|_{B^s_{p,q}(A_V)} 
& 
= \| \psi (A_V) f \|_{L^p} 
  + \Big\{ \sum _{ j \in \mathbb N} 
        \Big( 2^{sj} \| \Phi_j (\sqrt{A_V})\phi_j (\sqrt{A_V} ) f \|_{L^p}
        \Big)^q
    \Big\}^{\frac{1}{q}}
\\
& 
\leq C \| f \|_{L^1} 
  + C \Big\{ \sum _{ j \in \mathbb N} 
        \Big( 2^{sj} 2^{n(1-\frac{1}{p})j} 2^{-Mj} 
              \cdot 2^{Mj} \| \phi_j (\sqrt{A_V} ) f \|_{L^1}
        \Big)^q
    \Big\}^{\frac{1}{q}}
\\
& 
\leq C p_{V,M} (f) 
  + C \Big\{ \sum _{ j \in \mathbb N} 
        \Big( 2^{sj} 2^{n(1-\frac{1}{p})j} 2^{-Mj} 
        \Big)^q
    \Big\}^{\frac{1}{q}} 
    p_{V, M}  (f) 
\\
& 
\leq C p_{V,M}(f) 
\end{split}
\end{equation}
for any $f \in \mathcal X_V(\Omega)$. Thus we get the first embedding:
\begin{equation}\label{EQ:X-emb}
\mathcal X_V(\Omega) \hookrightarrow B^s_{p,q} (A_V).
\end{equation}
To prove the second embedding 
\begin{equation}\label{EQ:second}
B^s_{p,q} (A_V) \hookrightarrow \mathcal X'_V(\Omega),
\end{equation}
we take $M' \in \mathbb N$ such that 
$M' > -s + n(1-1/p')$. 
Applying Lemma \ref{lem:decomposition1} {\rm (i)}, 
the identities $\psi = \Psi \psi$, $\phi_j = \Phi_j \phi_j$, 
H\"older's inequality and 
the embedding \eqref{EQ:X-emb} 
for $s,p,q$ replaced by $-s, p^\prime,q^\prime$, i.e., 
$$
\mathcal X_V(\Omega) \hookrightarrow B^{-s}_{p',q'} (A_V), 
$$
we have, for 
 $f \in B^s_{p,q} (A_V)$ and $g \in \mathcal X_V(\Omega)$
\begin{equation}\notag 
\begin{split}
|_{\mathcal X_V'} \langle f, g\rangle_{\mathcal X_V} | 
& 
= 
\Big| \,_{\mathcal X'_V} \big\langle \psi(A_V) f, \Psi(A_V) g \big\rangle_{\mathcal X_V}  
+ \sum _{ j \geq 1} 
 \, _{\mathcal X_V'} \big\langle \phi_j (\sqrt{A_V}) f, \Phi_j (\sqrt{A_V}) g \big\rangle_{\mathcal X_V} 
 \Big| 
\\
& 
\leq \| \psi (A_V) f \|_{L^p} \| \Psi (A_V) g \|_{L^{p'}}   
\\
& \quad 
  + \big\| \big\{ 2^{sj} \| \phi_j (\sqrt{A_V}) f \|_{L^p} \big\}_{j \in \mathbb N} \big\|_{\ell^q (\mathbb N)}
    \big\| \big\{ 2^{-sj} \| \Phi_j (\sqrt{A_V}) g \|_{L^{p'}} \big\}_{j \in \mathbb N} \big\|_{\ell^{q'} (\mathbb N)}
\\
& 
\leq C \| f \|_{B^s_{p,q}(A_V)} \| g \|_{B^{-s}_{p' , q'}(A_V)}
\\
& 
\leq C \| f \|_{B^s_{p,q} (A_V)} p_{M'} (g). 
\end{split}
\end{equation}
Therefore, \eqref{EQ:second} is proved, and as a result, we get the embedding \eqref{902-1}. \\

Next we show the embedding \eqref{902-2}, namely, 
$$
\mathcal Z_V (\Omega) \hookrightarrow 
\dot B^s_{p,q} (A_V) \hookrightarrow 
\mathcal Z_V '(\Omega).
$$ 
Put 
$$
\Phi _j := \phi_{j-1} + \phi_j + \phi_{j+1} 
\quad \text{for all } j \in \mathbb Z . 
$$
Let $L \in \mathbb N$ be such that 
$L > |s| + n (1- 1/p)$. 
For any $f \in \mathcal Z (\Omega)$, 
we deduce from the identity $\phi_j = \Phi_j \phi_j$ and 
the estimate \eqref{homog} for $\alpha = 0$ that 
\begin{equation}\notag 
\begin{split}
\| f \|_{\dot B^s_{p,q}(A_V)} 
& 
=    \Big\{ \sum _{ j \in \mathbb Z} 
          \Big( 2^{sj} \| \Phi_j (\sqrt{A_V})\phi_j (\sqrt{A_V}) f \|_{L^p}
          \Big) ^q
     \Big\} ^{\frac{1}{q}}
\\
& 
\leq C \Big\{ \Big( \sum _{ j \leq 0} + \sum _{j \geq 1} \Big) 
          \Big( 2^{sj} 2^{n(1-\frac{1}{p})j} \| \phi_j (\sqrt{A_V}) f \|_{L^1}
          \Big) ^q
     \Big\} ^{\frac{1}{q}}
\\
& 
\leq C \Big( \sup _{ j \leq 0} 2^{-Lj} \| \phi_j (\sqrt{A_V}) f \|_{L^1} \Big) 
     \Big\{ \sum _{j \leq 0} 
          \Big( 2^{sj} 2^{n(1-\frac{1}{p})j} 2^{Lj} 
          \Big) ^q
     \Big\} ^{\frac{1}{q}}
\\
& 
\quad + C  
\Big( \sup _{ j \geq 1} 2^{Lj} \| \phi_j (\sqrt{A_V}) f \|_{L^1} \Big) 
     \Big\{ \sum _{j \geq 1} 
          \Big( 2^{sj} 2^{n(1-\frac{1}{p})j} 2^{-Lj} 
          \Big) ^q
     \Big\} ^{\frac{1}{q}}
\\
& 
\leq C q_{V,L} (f), 
\end{split}
\end{equation}
which implies that
\begin{equation}\label{EQ:Z-emb}
\mathcal Z_V(\Omega) \hookrightarrow \dot B^s_{p,q} (A_V).
\end{equation}
To prove the second embedding 
$$
\dot B^s_{p,q} (A_V) \hookrightarrow 
\mathcal Z_V '(\Omega),
$$ 
we take $L' \in \mathbb N$ such that 
$L' > |s| + n (1- 1/p')$. 
For any $f \in \dot B^s_{p,q} (A_V)$ and $g \in \mathcal Z_V(\Omega)$, 
using the identities $\phi_j = \Phi_j \phi_j$, 
H\"older's inequality and 
the embedding \eqref{EQ:Z-emb} for 
$s,p,q$ replaced by $-s,p^\prime,q^\prime$, i.e., 
$$
\mathcal Z_V(\Omega) \hookrightarrow \dot B^{-s}_{p',q'} (A_V),
$$ 
we estimate 
\begin{equation}\notag 
\begin{split}
|_{\mathcal Z_V'} \langle f, g\rangle_{\mathcal Z_V} | 
& 
= 
\Big| \sum _{ j \in \mathbb Z} 
 \, _{\mathcal Z_V'} \big\langle \phi_j (\sqrt{A_V}) f, \Phi_j (\sqrt{A_V}) g \big\rangle_{\mathcal Z_V} 
 \Big| 
\\
& 
\leq 
    \big\| \big\{ 2^{sj} \| \phi_j (\sqrt{A_V}) f \|_{L^p} \big\}_{j \in \mathbb Z} \big\|_{\ell^q (\mathbb Z)}
    \big\| \big\{ 2^{-sj} \| \Phi_j (\sqrt{A_V}) g \|_{L^{p'}} \big\}_{j \in \mathbb Z} \big\|_{\ell^{q'} (\mathbb Z)}
\\
& 
\leq C \| f \|_{\dot B^s_{p,q}(A_V)} \| g \|_{\dot B^{-s}_{p' , q'}(A_V)}
\\
& 
\leq C \| f \|_{\dot B^s_{p,q}(A_V)} q_{L'} (g). 
\end{split}
\end{equation}
Thus we conclude \eqref{902-2}.
\end{proof}

It remains to show that $B^s_{p,q } (A_V)$ and $\dot B^s_{p,q} (A_V)$ 
are Banach spaces. It is easy to check that they are normed vector spaces, 
and hence, it suffices to prove the completeness. 

\begin{proof}[{\bf Proof of the completeness of $B^s_{p,q} (A_V)$ and $\dot B^s_{p,q} (A_V)$. }] 
We have only to prove the completeness 
of the homogeneous Besov spaces $\dot B^s_{p,q} (A_V)$, 
since the inhomogeneous case is similar. 
The proof is done by the analogous argument to that by 
Triebel \cite{Triebel_1983}. 
Indeed, 
let $\{ f_N \}_{N=1}^\infty$ be a Cauchy sequence in $\dot B^s_{p,q} (A_V)$. 
We may assume that 
\begin{equation}\label{EQ:sub}
\| f_{N+1} - f_N\|_{\dot B^s_{p,q} (A_V)} \leq 2^{-N}
\end{equation}
without loss of generality. 
Then $\{ f_N \}_{N=1}^\infty$ is also a Cauchy sequence in 
$\mathcal Z_V '(\Omega)$ by the inclusion relation \eqref{902-2}, 
and hence, there exists an element $f \in \mathcal Z_V '(\Omega)$ with the property that
\begin{equation*} 
f_N \to f  \quad \text{in } \mathcal Z_V '(\Omega)
\quad \text{as } N \to \infty,
\end{equation*}
since $\mathcal Z_V '(\Omega)$ is complete. 
This together with the boundedness of $\phi_j (\sqrt{A_V})$ 
on $\mathcal Z'_V(\Omega)$ imply that 
\begin{equation}\label{1116-1}
\phi_j (\sqrt {A_V}) f_N \to \phi_j (\sqrt {A_V}) f \quad \text{ in } 
\mathcal{Z}^\prime_V (\Omega)
\text{ as } N \to \infty , 
\end{equation}
and we have $\phi_j (\sqrt{A_V}) f \in L^\infty (\Omega)$ by Lemma~\ref{lem:decomposition1} {\rm (ii)}. 
Furthermore, fixing $j\in \mathbb{Z}$, 
we see that $\{\phi_j (\sqrt {A_V} \, ) f_N\}_{N=1}^\infty$ is also a Cauchy sequence in $L^p(\Omega)$, 
and there exists $F_j \in L^p (\Omega)$ such that 
$$
\phi_j (\sqrt {A_V} \, ) f_N \to F_j 
\quad \text{in } L^p (\Omega) 
\text{ as } N \to \infty ,
$$
which implies that 
\begin{gather}
\notag 
F_j(x) = \phi_j (\sqrt{A_V}) f(x) \quad 
\text{almost every } x \in \Omega, 
\end{gather}
and the convergence \eqref{1116-1} also holds in the topology of 
$L^p (\Omega)$.

It remains to show that  
$f \in \dot B^s_{p,q} (A_V) $ 
and $f_N$ tends to $f$ in $\dot B^s_{p,q} (A_V)$ 
for the above $f \in \mathcal Z_V' (\Omega) $.  
Since $ \big\{ \{ 2^{sj } \| \phi_j (\sqrt {A_V}) f_N \|_{L^p} 
   \}_{j \in \mathbb Z} \big\}_{N=1}^\infty$ 
is a Cauchy sequence in $\ell ^q (\mathbb Z)$ and 
$$
 2^{sj } \| \phi_j (\sqrt {A_V}) f_N \|_{L^p} 
 \to  2^{sj } \| \phi_j (\sqrt {A_V}) f \|_{L^p}  \quad \text{as } N\to \infty, 
$$
we get 
$$
\| f \|_{\dot B^s_{p,q}(A_V)} < \infty, 
$$
and hence, 
$$
f \in \dot B^s_{p,q} (A_V). 
$$
For the convergence of $f_N$ to $f$, 
writing 
$$
f 
=\sum _{k = 1} ^\infty (f_k - f_{k-1})
=
\lim _{ N \to \infty}f_N 
\quad \text{in } \mathcal Z'_{V}(\Omega), 
$$
where $f_{0} = 0$, 
we conclude 
from \eqref{EQ:sub} that 
the above series converges absolutely in the topology of $\dot B^s_{p,q} (A_V)$. 
Thus the completeness of $\dot B^s_{p,q} (A_V)$ is proved. The proof of Theorem \ref{thm:1} is now finished. 
\end{proof}

\section{Proof of Proposition \ref{thm:2}}

In this section we prove Proposition \ref{thm:2}.
We treat only the homogeneous Besov spaces 
$\dot B^s_{p,q} (A_V)$,
since the inhomogeneous case follows analogously. 
We prove that 
\begin{equation}\label{EQ:first inc}
\dot B^s_{p,q} (A_V) ^* = \dot B^{-s}_{p',q'} (A_V) 
\end{equation} 
for any $s\in \mathbb{R}$ and $1\leq p,q < \infty$. 
Let us first show that 
\begin{equation}\label{EQ:B-inc1}
\dot B^{-s}_{p',q'} (A_V)
\hookrightarrow
\dot B^s_{p ,q} (A_V) ^*.
\end{equation}
Let $\{ \phi_j \}_{j \in \mathbb Z}$ be as in \eqref{917-2} 
and put 
$$ 
  \Phi_j := \phi_{j-1} + \phi_j + \phi_{j+1} 
 \quad \text{for } j \in \mathbb Z .
$$ 
For any $f \in \dot B^{-s}_{p',q'} (A_V)$, 
we define an operator $T_f$ as
$$
T_f g
:= \sum _{ j \in \mathbb Z} 
  \int_{\Omega} \Big(\phi_j (\sqrt {A_V}) f\Big)  \,\,
       \overline{ \Phi_j (\sqrt{A_V} ) g} \, dx  
    \quad \text{for } g \in \dot B^s_{p,q} (A_V).
$$
Then 
\begin{equation}\notag 
\begin{split}
|T_f g| 
& 
\leq 
\big\| \{ 2^{-sj} \| \phi_j (\sqrt {A_V}) f \|_{L^{p'}} \}_{j \in \mathbb Z}
\big\|_{\ell^{q'} (\mathbb Z)}
\big\| \{ 2^{sj} \| \Phi_j (\sqrt {A_V}) g \|_{L^p} \}_{j \in \mathbb Z}
\big\|_{\ell^q (\mathbb Z)}
\\
& 
\leq C \| f \|_{\dot B^{-s}_{p',q'}(A_V)} \| g \|_{\dot B^s_{p,q}(A_V)},
\end{split}
\end{equation}
which implies that
the operator norm 
$\| T_f  \|_{ \dot B^s_{p,q} (A_V) ^*} $
is bounded by $C\| f \|_{\dot B^{-s}_{p',q'} (A_V)}$. 
This proves 
the inclusion \eqref{EQ:B-inc1}.

We prove the converse inclusion:
\begin{equation}\label{EQ:B-inc2}
\dot B^s_{p ,q} (A_V) ^*
\hookrightarrow
\dot B^{-s}_{p',q'} (A_V).
\end{equation}
Let $F \in  \dot B^s_{p,q} (A_V) ^*$. We define an operator 
$$T : \ell^q ( \mathbb Z \, ; L^p (\Omega) ) \to \mathbb C
$$ 
as follows. 
For $G = \{ G_j \}_{j \in \mathbb Z}
\in \ell^q ( \mathbb Z \, ; L^p (\Omega) )$, 
we put 
$$
T (G) 
:= F\Big(\sum _{j \in \mathbb Z} 2^{-sj} \phi_j (\sqrt{A_V}) G_j \Big). 
$$
Here we estimate 
\begin{equation}\notag 
\begin{split}
& 
\Big\| \sum _{j \in \mathbb Z} 2^{-sj} \phi_j (\sqrt{A_V}) G_j 
    \Big\|_{\dot B^s_{p,q} (A_V)}
\\
=
& 
 \Big\{ \sum _{ k \in \mathbb Z} 
        \Big( 2^{sk} 
              \Big\| \phi_k (\sqrt{A_V}) \sum _{j = k-1 }^{k+1}  2^{-sj} \phi_{j} (\sqrt{A_V}) G_j 
              \Big\|_{L^p}
        \Big)^q 
  \Big\} ^{\frac{1}{q}}
\\
=
& 
 \Big\{ \sum _{ k \in \mathbb Z} 
        \Big( 2^{sk} 
              \Big\| \phi_k (\sqrt{A_V}) \sum _{r = -1 }^{1}  2^{-s(k+r)} \phi_{k+r} (\sqrt{A_V}) G_{k+r} 
              \Big\|_{L^p}
        \Big)^q 
  \Big\} ^{\frac{1}{q}}
\\
\leq
& 
C\sum _{r = -1 }^{1} 2^{-sr}
 \Big\{ \sum _{ k \in \mathbb Z} 
              \| G_k \|_{L^p} ^q 
  \Big\} ^{\frac{1}{q}}
\\
\leq 
& C \| G \|_{\ell^q L^p} ,
\end{split}
\end{equation}
where we used the estimate \eqref{homog} for $\alpha = 0$ in Lemma \ref{lem:calc}. 
Hence we deduce that
\begin{align*} 
| T(G) | 
\leq& \| F \|_{\dot B^s_{p,q} (A_V) ^*}
    \Big\| \sum _{j \in \mathbb Z} 2^{-sj} \phi_j (\sqrt{A_V}) G_j 
    \Big\|_{\dot B^s_{p,q} (A_V)}
\\
\leq& C \| F \|_{\dot B^s_{p,q} (A_V) ^*}
    \| G \|_{\ell ^q L^p}. 
\nonumber
\end{align*}
Since $(\ell ^q L^p)^* = \ell^{q'} L^{p'}$, 
there exists $\{ F_j \}_{j \in \mathbb Z} \in \ell ^{q'} L^{p'}$ such that 
\begin{equation}\label{821-1}
T(G) = 
\sum _{ j \in \mathbb Z} \int_{\Omega} F_j (x) \overline{G_j (x)} \, dx 
\quad \text{and} \quad 
\| \{ F_j \}_{j \in \mathbb Z} \|_{\ell ^{q' } L^{p'}} 
\leq C \| F \|_{\dot B^s_{p,q} (A_V)^*} .
\end{equation}
Then for any $g \in \dot B^s_{p,q} (A_V)$, 
let us take $G = \{ G_j \}_{j \in \mathbb Z}$ as 
$$
G _j = 2^{sj} \Phi_j (\sqrt{A_V}) g. 
$$
It follows from $g \in \mathcal Z_V ' (\Omega)$, Lemma~\ref{lem:decomposition1} {\rm (ii)} and 
the identities $\phi_j = \phi_j \Phi_j$ that
\begin{equation}\notag 
\begin{split}
F (g) 
& = F \Big( \sum _{ j \in \mathbb Z} 2^{-sj} \phi_j (\sqrt{A_V} ) 
            \big( 2^{sj}\Phi_j (\sqrt{A_V}) g\big) \Big)
\\
& 
=  T(G) 
\\
& 
= \sum _{ j \in \mathbb Z} \int_{\Omega} F_j (x) \overline{G_j (x)}\,dx
\\
& 
= \sum_{j \in \mathbb Z} \int_{\Omega} 
     F_j(x)  \overline{ 2^{sj}\Phi_j(\sqrt{A_V}) g} \, dx
\\
& 
= \sum_{j \in \mathbb Z} \int_{\Omega} 
    \Big( 2^{sj} \Phi_j(\sqrt{A_V}) F_j(x) \Big) \overline{g} \, dx .
\end{split}
\end{equation}
Taking $f$ as 
$$
f = \sum_{j \in \mathbb Z} 2^{sj} \Phi_j(\sqrt{A_V}) F_j , 
$$
we deduce from \eqref{821-1} that 
\begin{equation}\notag 
\begin{split}
\| f \|_{\dot B^{-s}_{p',q'}(A_V)} 
 \leq 
 & 
 C \| \{ F_j \}_{j \in \mathbb Z} \|_{\ell^{q'}L^{p'}} 
\\
\leq 
& 
C \| F \|_{\dot B^s_{p,q} (A_V) ^*}, 
\end{split}
\end{equation}
which implies that
$f \in \dot B^{-s}_{p',q'} (A_V)$. 
Hence $F$ is regarded as an element in $\dot B^{-s}_{p',q'} (A_V)$, 
and we get the inclusion
\eqref{EQ:B-inc2}; thus we 
conclude the isomorphism \eqref{EQ:first inc}. 
This ends the proof of Proposition \ref{thm:2}.

\section{Proof of Proposition \ref{thm:3}}

In this section we prove Proposition \ref{thm:3}. 
The embedding relations are immediate consequences of Lemma~\ref{lem:calc}. 
The main point is to prove the lifting properties. 

\vskip3mm 

First we prove the homogeneous case, namely, 
$$
A_V ^{s_0/2} f \in \dot B^{s-s_0}_{p,q} (A_V)
\quad \text{for any } f \in \dot B^s_{p,q} (A_V) . 
$$
To begin with, we show that 
\begin{equation}\label{0204-5}
A_V ^{s_0/2} 
\text{ is a continuous operator from }
\mathcal Z_V '(\Omega)
\text{ to itself.}
\end{equation}
By the definition \eqref{901-5}, it is sufficient to verify 
that $A_V ^{s_0 /2}$ is the continuous operator from $\mathcal Z_V (\Omega)$ to itself. 
Let us take 
$M_0 \in \mathbb N$ such that $M_0 > |s_0|$. 
It follows from \eqref{homog} 
for $\alpha = s_0/2$ and \eqref{907-2} that 
\begin{equation}\notag 
q_{V,M} \big( A_V ^{s_0 /2} g \big) 
\leq C q_{V,M+M_0} (g) 
\end{equation}
for any $g \in \mathcal Z_V (\Omega)$, 
which implies that $A_V ^{s_0 /2} g \in \mathcal Z_V (\Omega)$.
 This proves \eqref{0204-5}. 
Hence, all we have to do is to prove that 
$f \in \dot B^s_{p,q} (A_V) $ satisfies 
\begin{equation} \label{EQ:Required} 
\| A_V ^{s_0 /2} f \|_{\dot B^{s-s_0}_{p,q} (A_V)} 
\leq C \| f \|_{\dot B^s_{p,q} (A_V)}. 
\end{equation}
In fact, let 
$$
 \Phi _j := \phi_{j-1} + \phi_j + \phi_{j+1}.
$$%
We note that $\Phi_j (\lambda) \lambda ^{s_0}\in C_0 ^\infty ( (0,\infty))$. 
Writing
\[ 
\Phi_j (\lambda) \lambda ^{s_0} = 2^{s_0 j} 
\cdot\Phi_j ( \lambda ) \cdot (2^{-s_0j} \lambda^{s_0}), 
\]
we get 
\begin{align*}
\| \phi_j (\sqrt{A_V})A_V ^{s_0 /2} f \|_{L^p} 
=& 
 2^{s_0 j} \big\| \big\{ \Phi_j (\sqrt{A_V})2^{-s_0 j} A_V ^{s_0 /2} \big\} \phi_j (\sqrt{A_V}) f \big\|_{L^p} 
\\
\leq 
& C  2^{s_0 j} \| \phi_j (\sqrt{A_V}) f \|_{L^p}.
\end{align*}
Hence, multiplying $2^{(s-s_0)j}$ to the above inequality and taking 
the $\ell ^q (\mathbb Z)$-norm, we obtain the required inequality \eqref{EQ:Required}. 

As to inhomogeneous case, 
we have to consider the operators 
$$
(\lambda_0 ^2 + 1 + A_V) ^{s_0 /2} \phi_j (\sqrt{A_V}) . 
$$
The only different point from the homogeneous case is to show the following estimates:
\begin{equation}\label{129-1}
\left\| (\lambda_0 ^2 + 1 + A_V) ^{s_0 /2} \phi_j (\sqrt{A_V})f 
 \right\|_{L^p} 
 \leq C 2^{ s_0 j} \left\| \phi_j (\sqrt{A_V})f  \right\|_{L^p} 
\end{equation}
for any $j \in \mathbb N$. We write 
\begin{equation}\notag 
\begin{split}
(\lambda_0 ^2 + 1 + A_V) ^{s_0 /2}
= 
& \Big[ 2 ^{s_0 j} \big\{ 2^{-2j} ( \lambda_0 ^2 + 1) + 2^{-2j}A_V \big\} ^{s_0 /2} 
  - 2^{s_0 j} \big( 2^{-2j} A_V \big) ^{s_0 /2} 
 \Big] 
\\
& +  2^{s_0 j} \big( 2^{-2j} A_V \big) ^{s_0 /2} 
\\
=: 
& T_{1} + T_{2}. 
\end{split}
\end{equation}
As to $T_2 \phi_j (\sqrt{A_V})f $, 
it follows from \eqref{inhom} for 
$\alpha = s_0 /2$ in Lemma \ref{lem:calc} that 
$$
\| T_2 \phi_j (\sqrt{A_V})f \|_{L^p} 
\leq C 2^{s_0 j} \| \phi_j (\sqrt{A_V})f  \|_{L^p} . 
$$
Writing
\begin{equation}\notag
\begin{split}
T_{1} =
& 
2^{s_0 j} \int _0 ^{2^{-2j}  (\lambda _0^2 + 1)} 
  \partial_\theta(\theta + 2^{-2j} A_V) ^{\frac{s_0}{2} } \,d\theta
\\
=
& 2^{s_0 j} \int _0 ^{2^{-2j} (\lambda _0^2 + 1)} 
  \frac{s_0}{2} (\theta + 2^{-2j} A_V) ^{\frac{s_0}{2} -1} \,d\theta, 
\end{split}
\end{equation}
we estimate $T_{1 } \phi_j (\sqrt{A_V})f$ as 
\begin{equation}\notag
\begin{split}
\left\| T_{1} \phi_j (\sqrt{A_V})f
\right\|_{L^p}
\leq 
 C 2^{s_0 j} \int _0 ^{2^{-2j} (\lambda _0^2 + 1)} 
    \left\| (\theta + 2^{-2j} A_V) ^{\frac{s_0}{2} -1} \phi_j (\sqrt{A_V})f
    \right\|_{L^p} 
    \,d\theta  .
\end{split}
\end{equation}
When $p = 2$, we use the spectral theorem on the Hilbert space $L^2 (\Omega)$ to obtain 
\begin{equation}\notag
\begin{split}
\left\| (\theta + 2^{-2j} A_V) ^{\frac{s_0}{2} -1} \phi_j (\sqrt{A_V})f
    \right\|_{L^2} ^2 
= 
& 
 \int_{2^{2(j-1)}} ^{2^{2(j+1)}} 
   (\theta + 2^{-2j} \lambda ) ^{s_0 -2}
  \, d\left\| E_{A_V} (\lambda) \phi_j (\sqrt{A_V})f \right \|_{L^2} ^2
\\
\leq 
& 
 C \int_{2^{2(j-1)}} ^{2^{2(j+1)}} 
   ( 2^{-2j} \lambda ) ^{s_0 -2}
  \, d\left\| E_{A_V} (\lambda) \phi_j (\sqrt{A_V})f \right \|_{L^2} ^2
\\
\leq 
& C \left\| \phi_j (\sqrt{A_V})f \right \|_{L^2} ^2 ,
\end{split}
\end{equation}
since $j \in \mathbb N$ and $0 \leq \theta \leq 2^{-2j} (\lambda _0^2 + 1)$. 
When $p \not = 2$, we have to
obtain the following estimate: 
\begin{equation}\label{EQ:OB}
\left\| (\theta + 2^{-2j} A_V) ^{\frac{s_0}{2} -1} \phi_j (\sqrt{A_V})f
    \right\|_{L^p} 
\leq C \left\| \phi_j (\sqrt{A_V})f \right \|_{L^p} .
\end{equation}
Since $\theta $ is small compared with the spectrum of $2^{-2j} A_V \phi_j (\sqrt{A_V})$, 
$\theta$ is able to be neglected. 
Hence,  the proof of estimate \eqref{EQ:OB} 
is done by the argument of our paper 
\cite{IMT-preprint}. 
So, we may omit the details. Summarizing the estimates obtained now, we conclude the estimate \eqref{129-1}. 
The proof of Proposition \ref{thm:3} is finished.


\section{ Proofs of Propositions \ref{thm:4} and \ref{thm:7} }

In this section we prove Propositions \ref{thm:4} and \ref{thm:7}.
Let us start by preparing two lemmas. 

\begin{lem}\label{lem:62}
Let $1<p\le 2$. Then there exists a constant 
$C > 0$ such that 
\begin{gather}\label{901-8inhom}
\| f \|_{B^0_{p,2}(A_V)} 
\leq C \| f \|_{L^p} 
    + C \big\| \big\{ \| e^{- 2^{-2j}A_V}  f \|_{L^p } 
          \big\}_{j \in \mathbb N}
   \big\|_{\ell ^2 (\mathbb N)} , 
\\ \label{901-8}
\| f \|_{\dot B^0_{p,2}(A_V)} 
\leq C \big\| \big\{ \| e^{- 2^{-2j}A_V}  f \|_{L^p } 
          \big\}_{j \in \mathbb Z}
   \big\|_{\ell ^2 (\mathbb Z)} 
\end{gather}
for any $f\in C^\infty_0(\Omega)$.
\end{lem}

\begin{proof}
[\bf Proof.]
Since $e^{2^{-2j} \lambda ^2} \phi_j (\lambda)$ is in $C_0 ^\infty ((0,\infty))$, 
it follows from \eqref{inhom} for 
$\alpha = 0$ that 
\begin{align}\label{EQ:EQZ}
\| \phi _j (\sqrt{A_V}) f \|_{L^p} 
= &
\big\| \big(e^{2^{-2j} A_V} \phi _j (\sqrt{A_V}) \big) e^{-2^{-2j} A_V}f \|_{L^p} \\
\leq &
C \| e^{-2^{-2j} A_V} f \|_{L^p} 
\nonumber
\end{align}
for any $j \in \mathbb N$. Then, taking the $\ell^2 (\mathbb N)$-norm,
we obtain \eqref{901-8inhom}.
As to the homogeneous case, 
thanks to \eqref{homog} for $\alpha = 0$, inequality \eqref{EQ:EQZ} is also valid for any $j\in \mathbb{Z}$, 
and hence, taking the $\ell^2 (\mathbb Z)$-norm,
we conclude \eqref{901-8}. 
\end{proof}

\begin{lem}[The Khinchine inequality] \label{lem:Khinchine}
Let $\{ r_j (t) \}_{j=1}^\infty$ be a sequence of Rademacher functions, that is, 
$$
r_j (t) := 
\sum _{ k=1} ^{2^j} (-1)^{k-1} 
  \chi_{[(k-1)2^{-j} , k 2^{-j})} (t) 
  \quad \text{for } t \in [0,1],
$$
where $\chi_I$ denotes the characteristic function 
on the interval $I$. 
Then for any $p$ with $1 < p < \infty$, 
there exists a constant $C > 0$ such that 
\begin{equation}\label{901-11}
C^{-1} \| a \|_{\ell ^2 (\mathbb N)} 
\leq \Big\| \sum _{j \in \mathbb N} a_j r_j 
     \Big\| _{L^p (0,1)}
\leq C \| a \|_{\ell ^2 (\mathbb N)} 
\end{equation}
for all $a = \{ a_j \}_{j \in \mathbb Z} \in \ell^2 (\mathbb N)$. 
\end{lem}

\vskip3mm

\noindent 
{\bf Proof of Proposition \ref{thm:4} (i): The embedding  
  \begin{equation}\label{EQ:1-emb}
  L^p (\Omega) \hookrightarrow \dot B^0_{p,2} (A_V) \quad \text{for } 1 < p \leq 2.
 \end{equation} } 

\noindent 
It is sufficient to show that 
\begin{equation}\label{901-10}
\| f \|_{\dot B^0_{p,2}(A_V)} \leq C \| f \|_{L^p} 
\quad \text{for any } f \in C_0^\infty (\Omega)
\end{equation}
due to the fact that $C_0^\infty (\Omega)$ is dense in $L^p (\Omega)$. 
{Let $\{r_j(t)\}$ be the sequence of Rademacher functions as in Lemma \ref{lem:Khinchine}.}
If we show that there exists a constant 
$C > 0$ such that 
\begin{equation}\label{901-9}
\Big\| \sum _{ j = 1} ^N r_j (t) e^{-2^{-2j} A_V } f 
\Big\| _{L^p} 
+ 
\Big\| \sum _{j = -N } ^{-1} r_{-j} (t) e^{- 2^{-2j} A_V } f 
\Big\| _{L^p} 
\leq C \| f \|_{L^p}
\end{equation}
for all $t \in [0,1]$ and $N \in \mathbb N$,
then \eqref{901-10} is verified. 
Indeed, by using the Minkowski 
inequality, we have 
\begin{equation}\notag 
\begin{split}
& 
\Big(\sum _{ |j| \leq N} \| e^{ - 2^{-2j} A_V } f \|_{L^p} ^2 
\Big)^{1/2}
\\
\leq &
\| e^{ -A_V} f \| _{L^p}
 +   \Big\| \Big( \sum _{ j = 1 } ^N | e^{ -2^{-2j} A_V} f |^2 \Big )^{1/2}
     \Big\| _{L^p} 
 +   \Big\| \Big( \sum _{ j= -N }^{-1} | e^{ -2^{-2j} A_V} f |^2 \Big )^{1/2}
     \Big\| _{L^p} . 
\end{split}
\end{equation}
Since 
$\| e^{-A_V} f \| _{L^p} \leq C \| f \|_{L^p} 
$ 
by \eqref{1117-5}, 
and since the third term in the right member of the above estimate 
is treated analogously to the second one, 
we may consider only the second term. 
By using \eqref{901-11} and \eqref{901-9}, we estimate
\begin{equation}\notag 
\begin{split}
     \Big\| \Big( \sum _{ j=1}^N | e^{- 2^{-2j} A_V} f |^2 \Big )^{1/2}
     \Big\| _{L^p} 
\leq& 
  C   \Big\| \Big( \int_0^1 \Big| \sum _{j=1}^N r_j (t) e^{ -2^{-2j} A_V } f 
                           \Big| ^p
           \, dt \Big)^{1/p}
     \Big\| _{L^p}
\\
=& 
 C \Big( \int_0^1 
        \Big\| \sum _{j=1}^N r_j( t) e^{ - 2^{-2j} A_V} f \Big\| _{L^p}^p \, dt
  \Big) ^{1/p}
\\
\leq & 
C \Big( \int_0^1 \| f \|_{L^p} ^p \, dt \Big) ^{1/p}
\\
=& 
C \| f \|_{L^p}, 
\end{split}
\end{equation}
which implies that
$$
\Big(\sum _{ |j| \leq N} \| e^{ - 2^{-2j} A_V} f \|_{L^p} ^2 
\Big)^{1/2} 
\leq C \| f \|_{L^p} 
\quad \text{for any } N \in \mathbb N. 
$$
Taking the limit as $N \to \infty$ in the above inequality, and combining 
the resultant with the inequality \eqref{901-8} in 
Lemma \ref{lem:62}, 
we obtain the required inequality \eqref{901-10}. Thus, 
we get  
the embedding \eqref{EQ:1-emb}.\\

We must show \eqref{901-9}. 
Let $\widetilde f$  be the zero extension of $f$ to the outside of $\Omega$. 
Recall that
$G_t (x)$ is the function of Gaussian type 
in the right member of \eqref{1117-5}.
Noting that 
\[
|r_j(t)|\le 1 \quad \text{for all $t\in[0,1]$,}
\]
we deduce from \eqref{1117-5} 
that 
\begin{equation}\label{902-4}
\Big| \sum _{ j = 1} ^N r_j (t) e^{ - 2^{-2j} A_V} f
\Big| 
+ 
\Big| \sum _{ j = -N} ^{-1} r_j (t) e^{ - 2^{-2j} A_V } f
\Big| 
\leq C \int_{\mathbb R^n} \sum _{j=-N} ^N  G_{2^{-2j}} (x-y) |\widetilde f(y) | dy 
\end{equation}
for all $t\in[0,1]$.
Here, it is certain to check 
that for each $\alpha \in (\mathbb N\cup \{0\}
)^n$  
\begin{gather}\notag 
|x | ^{n + |\alpha|} \big| \partial_x ^\alpha G_{2^{-2j}} (x) \big|
\leq C |2^j x|^{n+|\alpha|} \big| \partial_x ^\alpha G_1 (2^j x) \big|, 
\end{gather}
 and hence, 
\begin{gather}\notag 
\begin{split}
\sup _{t \in [0,1] , N \in \mathbb N , x \in \mathbb R^n}|x | ^{n + |\alpha|} \Big| \sum _{j=-N} ^N \partial_x ^\alpha G_{2^{-2j}} (x)\Big| 
& \leq C \sum _{ j \in \mathbb Z} |2^j x|^{n+|\alpha|} \big| \partial_x ^\alpha G_1 (2^j x) \big|
\\
& \leq C \sum _{ j \in \mathbb Z} 2^{(n+|\alpha|)j} e^{-c2^{-2j}}
 < \infty .
\end{split}
\end{gather}
Then, applying the $L^p$-boundedness of the singular integral operators (see e.g. p.29 in \cite{Stein_1970}), we get 
\begin{equation}\notag 
\Big\| \int_{\mathbb R^n} \sum _{j=-N} ^N  G_{2^{-2j}} (x-y) |\widetilde f(y) | dy
\Big\|_{L^p(\mathbb R^n)} 
\leq C \| \widetilde f \|_{L^p (\mathbb R^n)}
= C \| f \|_{L^p } . 
\end{equation}
Hence the required inequality \eqref{901-9} is a consequence of 
\eqref{902-4} and the above estimate. 
Therefore, the proof of the embedding 
\eqref{EQ:1-emb} is completed. 
\hfill $\qed$

\vskip3mm 

\noindent 
{\bf Proof of Proposition \ref{thm:4} (ii): 
The embedding 
  \begin{equation}\label{EQ:2-emb}
  \dot B^0_{p,2} (A_V)
  \hookrightarrow 
  L^p (\Omega) \quad \text{for } 2 \leq p < \infty.
  \end{equation} }%
Let $p'$ be such that $1/p + 1/p' = 1$. 
Then the embedding \eqref{EQ:2-emb}
is an immediate consequence of $1 < p' \leq 2$, 
$L^{p'} (\Omega) \hookrightarrow \dot B^0_{p',2} (A_V)$, 
$ L^{p'}(\Omega) ^* = L^p (\Omega)$ and  
$\dot B^0_{p',2} (A_V)^*= \dot B^0_{p,2} (A_V)$.  
\hfill $\square$

\vskip3mm 

\noindent 
{\bf Proofs of Proposition \ref{thm:4} (i) and (ii) for the inhomogeneous Besov spaces. } 
Let $1 < p \leq 2$. Then all we have to do is to show that 
\begin{equation}\notag 
\| f \|_{B^0_{p,2}(A_V)} \leq C \| f \|_{L^p} 
\quad \text{for any } f \in C_0^\infty (\Omega)  . 
\end{equation}
Referring to the estimate \eqref{901-8inhom}, we have only to show the corresponding estimate 
to \eqref{901-9}, that is, 
\begin{equation}\notag 
\Big\| \sum _{ j = 1} ^N r_j (t) e^{-2^{-2j} A_V } f 
\Big\| _{L^p} 
\leq C \| f \|_{L^p},
\end{equation}
which is proved in 
the same way as in the proof of \eqref{901-9} by using the pointwise estimate 
\eqref{1117-4_2} for the kernel of $e^{-tA_V}$.
Hence we have the embedding 
\begin{equation}\label{EQ:rad} %
\text{$L^p (\Omega) \hookrightarrow B^0_{p,2} (A_V)$ \quad for $1 < p \leq 2$.}
\end{equation}

Finally, referring to the proof of \eqref{EQ:2-emb}, 
we obtain the embedding 
$$
B^0_{p,2} (A_V) \hookrightarrow L^p (\Omega) \quad \text{for } 2 \leq p < \infty 
$$
by taking the duality of 
$L^{p'} (\Omega) \hookrightarrow B^0_{p',2} (A_V)$.
The proof of Proposition \ref{thm:4} is now finished.
\hfill $\square$
\\

We now turn to the proof of Proposition \ref{thm:7}. 
\\

\noindent 
{\bf Proof of Proposition \ref{thm:7}. } 
Putting 
$$
\dot X ^s_{p,q} (A_V)
:=
\Big\{ f \in \mathcal X'_V(\Omega) 
     \, \Big| \, 
     \| f \|_{\dot B^s_{p,q}(A_V)} < \infty , \,
     f = \sum _{ j \in \mathbb Z} 
        \phi_j (\sqrt{A_V} \,) f 
        \text{ in } \mathcal X'_V(\Omega)
   \Big\},
$$
we see that 
$$ 
\dot X^s_{p,q} (A_V) \subset \dot B^s_{p,q} (A_V). 
$$
Hence it is sufficient to
prove that
\begin{equation}\label{EQ:AV-Besov}
\dot B^s_{p,q} (A_V) \hookrightarrow \dot X^s_{p,q} (A_V).
\end{equation}
Let $f \in \dot B^s_{p,q} (A_V)$. Then $f \in \mathcal Z_V'(\Omega)$, and
thanks to Lemma \ref{lem:decomposition1} (ii),  
$f$ is written as
\begin{equation}\label{1126-2}
\begin{split}
f  
& =  
\sum _{ j \leq 0} \phi_j (\sqrt{A_V}) f 
 + \sum _{j \geq 1} \phi_j ( \sqrt{A_V}) f  
 \quad \text{in } \mathcal Z_V ' (\Omega)
\\
& =:
  I + II . 
\end{split}
\end{equation}
For the low frequency part, it follows from \eqref{homog} for $\alpha = 0$ that 
\begin{align*} 
\| I \|_{L^\infty} 
\leq& \sum _{ j \leq 0 } \| \phi_j (\sqrt{A_V}) f \|_{L^\infty} \\
\leq& C \sum _{ j \leq 0 } 2^{\frac{n}{p}j} \| \phi_j (\sqrt{A_V}) f \|_{L^p} ,
\end{align*}
where the right member is finite when $(s,q) = (n/p , 1)$. 
In the case when $s < n /p$, we estimate 
\begin{align*} 
\| I \|_{L^\infty} 
\leq& C \sum _{ j \leq 0 } 2^{( \frac{n}{p} - s )j} 
    \sup _{ k \leq 0} 2^{s k } \| \phi_{ k} (\sqrt{A_V}) f \|_{L^p}\\
\leq& C \| f \|_{\dot B^s_{p,\infty} (A_V)}\\
\leq& C \| f \|_{\dot B^s_{p,q} (A_V)} , 
\end{align*}
\noindent 
where we used the embedding in Proposition \ref{thm:3} {\rm (ii)} in the last step.
Hence the above two estimates and Lemma \ref{cor:1} imply that $I$ 
belongs to $\mathcal X' _V (\Omega)$.  
As to $II$,  
since the high frequency part of $q_{V,M} ( \cdot )$ 
is equivalent to that of $p_{V,M}(\cdot)$, it follows that $II \in \mathcal X'_V (\Omega)$. 
Hence the identity \eqref{1126-2} holds in the topology of $\mathcal X'_V (\Omega)$. 
Therefore, we get %
$f \in \dot X^s_{p,q} (A_V)$. 
Thus we conclude the embedding \eqref{EQ:AV-Besov}.
This completes the proof of Proposition \ref{thm:7}. 
\hfill$\square$

\section{Proof of Proposition
\ref{thm:6}} \label{sec:thm:6}

In this section we prove Proposition 
\ref{thm:6}.
We utilize the theory of 
Lorentz spaces and introduce the following notations 
(see e.g. \cite{Grafakos_2014,Ziemer_1989}). 
Let $f$ be a measurable function on 
$\Omega$.
We define the non-increasing rearrangement of 
$f$ as 
\[
 f^* (t ) := 
 \inf \{ \,\, c > 0 \,\, | \,\, 
 m_f ( c ) \leq t \,\, \},  
\]
where 
$m_f ( c )$ 
is the distribution function of $f$ 
which is defined by the Lebesgue measure of the set 
$\{ \, x \in \mathbb R ^n \, | \, |f(x)| > 
c \, \}$. 
We define a function $f^{**}(t)$ on 
$(0,\infty)$ as  
\[
 f^{**}(t) := \frac{1}{t} 
  \int _0 ^t f^* (t') \, dt'. 
\]
Lorentz spaces $L^{p, q} (\Omega)$ 
are 
defined by letting 
\begin{equation} \notag
L^{p,q}( \Omega ) := 
 \{ \,\, f \, : \text{measurable on }  \Omega 
   \,\, |  \,\,  
 \| f \| _{L^{p , q} }< \infty \,\, \}, 
\end{equation}
where 
\begin{equation} \notag 
\| f \|_{L^{p,q} } 
 := 
 \begin{cases}\displaystyle 
 \Big\{ \int_0^\infty \big( t^{\frac{1}{p}} f^{**} (t) \big) ^q \frac{dt}{t} 
 \Big\} ^{\frac{1}{q}} 
 & \quad \text{if } 1 \leq p ,q < \infty,
 \\ \displaystyle 
 \sup _{t > 0 } 
      t ^{\frac{1}{p}} f^{**}(t). 
 & \quad \text{if } 1 \leq p \leq \infty , q = \infty.
 \end{cases}
\end{equation}
In what follows, 
we denote by $\| \cdot \|_{L^{p,q}(\mathbb R^n)}$ the norm of $L^{p,q} (\mathbb R^n)$ 
only when $\Omega = \mathbb R^n$. 
Note that 
\begin{gather}
\label{pq}
L^{p,1}(\Omega) \hookrightarrow L^{p,q} (\Omega)
\quad \text{if } 1 \leq p, q \leq \infty,
\\ \notag 
L^p (\Omega ) = L^{p,\infty} ( \Omega )  
\quad \text{if } p = 1,\infty,
\\ \notag 
L^{p,1} ( \Omega ) \hookrightarrow L^p ( \Omega ) = L^{p,p} (\Omega ) 
\hookrightarrow L^{p,\infty} ( \Omega )
\quad \text{if } 1 < p < \infty. 
\end{gather}
Let $1<p<\infty$. 
We have the H\"older inequality and Young inequality in the Lorentz spaces: 
\begin{align}
\label{Holder}
&\| f g \|_{L^{p,q}} 
\leq \| f \|_{L^{p_1 , q_1}} \| g \|_{L^{p_2, q_2}}  
\quad \text{if } \frac{1}{p} = \frac{1}{p_1} + \frac{1}{p_2}, \quad 
\frac{1}{q} = \frac{1}{q_1} + \frac{1}{q_2}, 
\\
\label{Holder1}
&\| f g \|_{L^1} 
\leq  \| f \|_{L^{p_1 , q_1}} \| g \|_{L^{p_2, q_2}} 
\hspace{6.5mm} 
 \text{if } 1 = \frac{1}{p_1} + \frac{1}{p_2} = \frac{1}{q_1} + \frac{1}{q_2}, 
\\ 
\label{Young}
& 
\| f * g \|_{L^{p,q} (\mathbb R^n)} 
\leq 
\| f \|_{L^{p_1 , q_1} (\mathbb R^n)} 
 \| g \|_{L^{p_2, q_2} (\mathbb R^n) }  
\\ \notag 
& 
\hspace{55.5mm} 
\text{if } 
\frac{1}{p} = \frac{1}{p_1} + \frac{1}{p_2} -1 , 
\quad \frac{1}{q} = \frac{1}{q_1} + \frac{1}{q_2} ,
\end{align}
where $1 \leq p_1,p_2,q,q_1,q_2 \leq \infty$. %
We often use the estimates in the Lorentz spaces 
on $\mathbb R^n$ for functions on $\Omega$ 
extending them by zero extension to the outside 
of $\Omega$ when the necessity arises. 
\\

We 
prove Proposition
\ref{thm:6}
only for the homogeneous Besov spaces $\dot B^s_{p,q} (A_V)$, 
since the inhomogeneous case is proved 
in an analogous way. 
\\

We prepare the following four lemmas. 

\begin{lem}\label{lem:91}
Let $1 \leq p_0 < p < \infty$ and $1 \leq q \leq \infty$. 
Assume that $V$ satisfies \eqref{1104-1} and \eqref{ass:1}. 
Then there exists a constant 
$C>0$ such that 
\begin{equation}\label{910-1}
\| \phi_j (\sqrt{A_V}) f \|_{L^{p,q}} 
+ \| \phi_j (\sqrt{A_0}) f \|_{L^{p,q}} 
\leq C 2^{n(\frac{1}{p_0} - \frac{1}{p})j}\| f \|_{L^{p_0}} . 
\end{equation}
for all $j \in \mathbb Z$ and $f \in 
L^{p_0} (\Omega)$. 
\end{lem}
\begin{proof}
[\bf Proof.] 
It is sufficient to consider the case $q = 1 $ due to the embedding \eqref{pq}. 
Let $p_1$ be such that $ 1/p = 1/p_0 + 1/p_1 -1$. 
Then it follows from the Young inequality \eqref{Young} and 
the same argument as Lemma~\ref{lem:calc}
that 
\begin{equation}
\begin{split}
\| \phi_j (\sqrt{A_V}) f \|_{L^{p,1}} 
& = \big\| e^{- 2^{-2j} A_V } \big\{ e^{2^{-2j}A_V} \phi_j (\sqrt{A_V}) \big\} f  
    \big\|_{L^{p,1}}
\\
& \leq \| G_{2^{-2j}} \|_{L^{p_1,1} (\mathbb R^n) } 
  \big\| \big\{ e^{2^{-2j}A_V} \phi_j 
  (\sqrt{A_V}) \big\} f  \big\|_{L^{p_0,\infty}}
\\
& 
\leq C(p_1) 2^{n(\frac{1}{p_0} - \frac{1}{p})j}
\big\| \big\{ e^{2^{-2j}\Delta} \phi_j 
(\sqrt{A_V}) \big\} f  \big\|_{L^{p_0}}
\\
& 
\leq C(p_1) 2^{n(\frac{1}{p_0} - \frac{1}{p})j}
\| f  \|_{L^{p_0}} ,
\end{split}
\end{equation}
where $G_t$ is the function of Gaussian type appearing in the right member of 
\eqref{1117-5},
and we used the fact that 
$$
\| G_{2^{-2j}} \|_{L^{p_1,1}} = C (p_1) 2^{n(\frac{1}{p_0} - \frac{1}{p})j} 
\quad \text{for } p_1 > 1.
$$
Here we note that the above constant $C = C(p_1)$ is finite if and only if $p_1 > 1$, 
and hence, we have to assume that $p_1 > 1$, namely, $p_0 < p$. 
The estimate for $\phi_j (\sqrt{A_0}) f$ is obtained in the same way. 
Thus the proof of  of Lemma~\ref{lem:91} is completed. 
\end{proof}

\begin{lem}\label{lem:add}
Let $\{ \phi_j \}_{j \in \mathbb Z}$ be defined by \eqref{917-2}.  
Assume that $V$ satisfies \eqref{1104-1}, \eqref{ass:1} and \eqref{ass:1_0}. 
Let $1 \leq p \leq \infty$. Then 
$$
A_V ^m \phi_j (\sqrt{A_0}) f \in \mathcal Z_V ' (\Omega) 
\quad \text{and} \quad 
A_0 ^m \phi_j (\sqrt{A_V}) f \in \mathcal Z_0 ' (\Omega) 
$$
for any $j, m \in \mathbb Z$ and $f \in L^p (\Omega)$. 
\end{lem}
\begin{proof}[\bf Proof] 
Let $j \in \mathbb Z$ be fixed. 
Since $\phi_j (\sqrt{A_0}) f \in L^ p (\Omega)$ for any $f \in L^p (\Omega)$ 
by \eqref{homog} for $\alpha = 0$ in Lemma \ref{lem:calc}, 
it follows from \eqref{0204-4} in Lemma \ref{cor:1} 
that $\phi_j (\sqrt{A_0}) f \in \mathcal Z_V ' (\Omega)$. 
We proved the assertion \eqref{0204-5} in the proof of Proposition~\ref{thm:3};  
$A_V^m$ is the mapping from $\mathcal Z_V'(\Omega)$ to itself. 
This proves the first assertion. 
In the same way, the second assertion holds. 
The proof of Lemma \ref{lem:add} is complete. 
\end{proof}

\begin{lem} \label{lem:1020}
Let $\{ \phi_j \}_{j \in \mathbb Z}$ be defined by \eqref{917-2}, 
and take $\Phi_j := \phi_{j-1} + \phi_j + \phi_{j+1}$. 
Assume that $V$ satisfies \eqref{1104-1}, \eqref{ass:1} and \eqref{ass:1_0}. 
Then the following assertions hold{\rm :} 
\begin{enumerate}
\item[(i)]
Let $p = 1$ for $n = 2$ and $1 \leq p < n/2$ for $n \geq 3$. 
Then we have, for any $f \in L^p (\Omega)$
\begin{equation}\label{908-1}
\| \phi_j (\sqrt{A_V}) \Phi_k (\sqrt{A_0}) f \|_{L^p} 
\leq C 2^{-2(j-k)} \| f \|_{L^p} ,
\end{equation}
\begin{equation} \label{908-2}
\| \phi_k (\sqrt{A_0}) \Phi_j (\sqrt{A_V}) f \|_{L^p} 
\leq C 2^{-2(k-j)} \| f \|_{L^p} .
\end{equation}

\item[(ii)]
Let $p = \infty$ for $n = 2$ and $n/(n-2) < p \leq \infty$ for $n \geq 3$. 
Then we have, for any $f \in L^p (\Omega)$
\begin{equation}\label{908-3}
\| \phi_j (\sqrt{A_V}) \Phi_k (\sqrt{A_0}) f \|_{L^p} 
\leq C 2^{-2(k-j)} \| f \|_{L^p} , 
\end{equation}
\begin{equation} \label{908-4}
\| \phi_k (\sqrt{A_0}) \Phi_j (\sqrt{A_V}) f \|_{L^p} 
\leq C 2^{-2(j-k)} \| f \|_{L^p} .
\end{equation}
\end{enumerate}
\end{lem}
\begin{proof}
[\bf Proof.] 
We prove only {\rm (i)}, since 
the estimates \eqref{908-3} and \eqref{908-4} are obtained by 
the duality argument for \eqref{908-2} and \eqref{908-1}, respectively. 

We first consider the case $ n =2$ and $p = 1$. 
We note from Lemma \ref{lem:add} that 
$$
\Phi_k (\sqrt{A_0 }) f = A_V^{-1} A_V \Phi_k (\sqrt{A_0 }) f 
\quad \text{in } \mathcal Z_V ' (\Omega). 
$$
Thanks to the estimate \eqref{homog} 
for $\alpha = 1$ and the assumption \eqref{ass:1_0} on $V$, 
a formal calculation implies that
\begin{equation}\label{910-2}
\begin{split}
\| \phi_j (\sqrt{A_V}) \Phi_k (\sqrt{A_0}) f \|_{L^1}
=& 
\| \phi_j (\sqrt{A_V}) A^{-1}_V A_V \Phi_k (\sqrt{A_0 }) f \|_{L^1}
\\
\leq & 
C 2^{ -2j} 
  \Big\{ \| A_0 \Phi_k (\sqrt{A_0}) f \|_{L^1}
       +\| V \Phi_k (\sqrt{A_0}) f \|_{L^1}
  \Big\}  
\\
\leq & 
C 2^{ -2j} 
  \Big\{ 2^{2k} \| f \|_{L^1}
       +\| V \|_{L^1} \| \phi_k (\sqrt{A_0}) f \|_{L^\infty}
  \Big\} 
\\
\leq & 
C 2^{ -2j} 2^{2k} \| f \|_{L^1},
\end{split}
\end{equation}
which proves \eqref{908-1}. 
As to the estimate \eqref{908-2}, 
again by using \eqref{homog} and the assumption \eqref{ass:1_0} on $V$, we estimate 
\begin{equation}
\begin{split}
\label{910-3}
& 
 \| \phi_k (\sqrt{A_0}) \Phi_j (\sqrt{A_V}) f 
\|_{L^1} 
\\
=& 
\| \phi_k (\sqrt{A_0}) A^{-1}_0 (A_V - V) \Phi_j (\sqrt{A_V}) f \|_{L^1} 
\\
\leq & 
 C 2^{-2k} 
    \Big\{ \| A_V \Phi_j (\sqrt{A_V}) f \|_{L^1} 
         + \| V \Phi_j (\sqrt{A_V}) f \|_{L^1} 
    \Big\}
\\
\leq& 
C 2^{-2k} 
    \Big\{ 2^{2j} \| f \|_{L^1} 
         + \| V \|_{L^1} \| \Phi_j (\sqrt{A_V}) f \|_{L^\infty} 
    \Big\}
\\
\leq& 
C 2^{-2k} 2^{2j} \| f \|_{L^1}. 
\end{split}
\end{equation}
This proves \eqref{908-2}.
Thus the estimate (i) for $n = 2$ and $p = 1$ is obtained. 

In the case when $n \geq 3$, we estimate by the use of the Lorentz spaces.  
As to the estimate \eqref{908-1}, by using the same argument as in \eqref{910-2},
we get 
\begin{equation}\label{EQ:F1} 
\begin{split}
\| \phi_j (\sqrt{A_V}) \Phi_k (\sqrt{A_0}) f \|_{L^p} 
 \leq C 2^{ -2j} 
   \big\{ 2^{2k} \| f \|_{L^p} 
         + \| V \Phi_k (\sqrt{A_0}) f \|_{L^{p}} 
   \big\}.
\end{split}
\end{equation}
Here, the H\"older inequalities \eqref{Holder1} and \eqref{Holder}
together with the estimate 
\eqref{910-1} in Lemma \ref{lem:91} imply that for $ p = 1$,
\begin{align}\label{910-4} 
\| V \Phi_k (\sqrt{A_0}) f \|_{L^{1}} 
\leq& \| V \|_{L^{\frac{n}{2},\infty}} \| \Phi_k (\sqrt{A_0}) f \|_{L^{\frac{n}{n-2},1}}\\
\leq& C 2^{2k} \| f \|_{L^1}, \nonumber
\end{align}
and for $p > 1$,
\begin{align}\label{910-5}
\| V \Phi_k (\sqrt{A_0}) f \|_{L^{p}} 
\leq& \| V \|_{L^{\frac{n}{2},\infty}} \| \Phi_k (\sqrt{A_0}) f \|_{L^{p_0,p}}\\
\leq& C 2^{2k} \| f \|_{L^p}, \nonumber
\end{align}
where $p_0$ is a real number with $1/p = 2/n + 1/p_0$. Then \eqref{908-1} is obtained 
by estimates \eqref{EQ:F1}--\eqref{910-5}. 
It remains to prove the estimate \eqref{908-2}.
By the same argument as in \eqref{910-3} we estimate
\begin{equation}\notag 
\begin{split}
\| \phi_k (\sqrt{A_0}) \Phi_j (\sqrt{A_V}) f \|_{L^p} 
\leq C 2^{-2k} 
    \Big\{ 2^{2j} \| f \|_{L^p} 
         + \| V \Phi_j (\sqrt{A_V}) f \|_{L^p} 
    \Big\} . 
\end{split}
\end{equation}
Here, it follows from the same argument as \eqref{910-4} and \eqref{910-5} that 
\begin{equation}\notag 
\| V \Phi_j (\sqrt{A_V}) f \|_{L^p} 
\leq C 2^{ 2j} \| f \|_{L^p}. 
\end{equation}
Then, \eqref{908-2} is 
a consequence of the above two estimates. 
The proof of Lemma \ref{lem:1020} is complete. 
\end{proof}

\begin{lem}\label{lem:1020-2}
Under the same assumptions as Lemma \ref{lem:1020}, the following assertions hold{\rm :} 
\begin{enumerate}
\item[(i)]
Let $1 \leq p < \infty$ and $0 \leq \alpha < \min \{2, n/p\}$. 
Then 
we have
\begin{equation}\label{1020-1}
\| \phi_j (\sqrt{A_V}) \Phi_k (\sqrt{A_0}) f \|_{L^p} 
\leq C 2^{-\alpha (j-k)} \| f \|_{L^p} ,
\end{equation}
\begin{equation} \label{1020-2}
\| \phi_k (\sqrt{A_0}) \Phi_j (\sqrt{A_V}) f \|_{L^p} 
\leq C 2^{-\alpha (k-j)} \| f \|_{L^p} .
\end{equation}
for any $j,k \in \mathbb Z$ and $f \in L^p (\Omega)$. 
\item[(ii)]
Let $1 < p \leq \infty$ and $0 \leq \alpha < \min \{2, n(1-1/p)\}$. 
Then we have
\begin{equation}\label{1020-3}
\| \phi_j (\sqrt{A_V}) \Phi_k (\sqrt{A_0}) f \|_{L^p} 
\leq C 2^{-\alpha (k-j)} \| f \|_{L^p} , 
\end{equation}
\begin{equation} \label{1020-4}
\| \phi_k (\sqrt{A_0}) \Phi_j (\sqrt{A_V}) f \|_{L^p} 
\leq C 2^{-\alpha (j-k)} \| f \|_{L^p} .
\end{equation}
for any $j,k \in \mathbb Z$ and $f \in L^p (\Omega)$. 
\end{enumerate}
\end{lem}

\begin{proof} 
[\bf Proof]
The strategy of the proof is to apply the 
Riesz-Thorin interpolation theorem 
to the estimates in Lemma \ref{lem:1020} and the following uniform estimates:
\begin{equation}\label{1020-5}
\| \phi_j (\sqrt{A_V}) \Phi_k (\sqrt{A_0}) f \|_{L^q}  
\leq C \| f \|_{L^q} , 
\end{equation}
\begin{equation}\label{EQ:1020-5}
\| \phi_k (\sqrt{A_0}) \Phi_j (\sqrt{A_V}) f \|_{L^q}  
\leq C \| f \|_{L^q} , 
\end{equation}
for all $j,k \in \mathbb Z$, 
 which are proved by \eqref{homog} 
for $\alpha = 0$. 

Let $0\leq \alpha<\min\{2,n/p\}$.
Then the proof of \eqref{1020-1}
for $1 \leq p < n/2$ is performed by combining \eqref{908-1} 
and \eqref{1020-5} with $q = p$.
In fact, we estimate  
\begin{align}\label{1020-6}
& 
\| \phi_j (\sqrt{A_V}) \Phi_k (\sqrt{A_0}) f \|_{L^p}
\\
=& \| \phi_j (\sqrt{A_V}) \Phi_k (\sqrt{A_0}) f \|_{L^p} ^{\frac{\alpha}{2}}
    \| \phi_j (\sqrt{A_V}) \Phi_k (\sqrt{A_0}) f \|_{L^p} ^{1-\frac{\alpha}{2}}
\nonumber \\
\leq & C \{2^{-2 (j-k)}\}^{\frac{\alpha}{2}} \| f \|_{L^p}
\nonumber \\
=& C 2^{-\alpha (j-k)} \| f \|_{L^p}.
\nonumber 
\end{align}
This proves \eqref{1020-1}.
In a similar way, by using \eqref{908-2} and 
\eqref{EQ:1020-5}, we get the estimate \eqref{1020-2}.
When $n/2 \leq p \leq \infty$, we apply the Riesz-Thorin interpolation theorem
to \eqref{1020-5} with $q = \infty$ and the estimate \eqref{908-1} 
together with the argument \eqref{1020-6}.

Finally, estimates \eqref{1020-3} and \eqref{1020-4} are proved in analogous way 
as in \eqref{1020-1} and \eqref{1020-2}, if we divide the cases into 
$n/(n-2) < p \leq \infty$ and $1 \leq p \leq n / (n-2)$. 
The proof of Lemma \ref{lem:1020-2} is complete. 
\end{proof}

\noindent
{\bf Remark.} 
When \eqref{ass:1} is not imposed on $V$, which is the assumption on the inhomogeneous Besov spaces,
the same estimates in Lemmas \ref{lem:91}--\ref{lem:1020-2} also hold for $j,k \in \mathbb N$, 
since the proof is done analogously by 
applying \eqref{inhom}, \eqref{1117-4_2} instead of 
\eqref{homog}, \eqref{1117-5}, respectively. 
\\

In what follows, we prove the equivalence relation between 
$\dot B^s_{p,q} (A_0)$ and $\dot B^s_{p,q} (A_V)$ 
under the assumption on $V$ in Proposition \ref{thm:6}.

\vskip3mm 

\noindent
{\bf Proof of the isomorphism:
\begin{equation}\label{EQ:Finalle}
\dot B^s_{p,q} (A_0) \cong \dot B^s_{p,q} (A_V). 
\end{equation}}

\noindent
{\bf The case: $s > 0$. }
First we prove that
\begin{equation}\label{EQ:Bes-1}
\dot B^s_{p,q} (A_0) \hookrightarrow \dot B^s_{p,q} (A_V)
\end{equation}
for any $s > 0$. 
To begin with, for any 
$f \in \dot B^s_{p,q} (A_0)$, we show that 
\begin{equation}\label{908-5}
f = \sum _{ j \in \mathbb Z} \phi_j (\sqrt{A_V}) f 
 \quad \text{in } \mathcal Z'_V (\Omega) . 
\end{equation}
To see \eqref{908-5}, 
we consider the formal identity 
\begin{equation}\label{EQ:formal-g}
_{\mathcal Z'_V}\langle f, g 
 \rangle _{\mathcal Z_V} 
 = \sum _{j \in \mathbb Z} \, _{\mathcal Z'_V} 
 \langle f, \phi_j (\sqrt{A_V}) g 
 \rangle _{\mathcal Z_V} 
 = \sum _{j \in \mathbb Z} \, _{\mathcal Z'_V} 
 \langle \phi_j (\sqrt{A_V}) f, g 
 \rangle _{\mathcal Z_V} ,
\end{equation}
where the first identity is deduced from 
Lemma \ref{lem:decomposition1} {\rm (ii)}. 
Note that 
\begin{equation}\label{0128-1}
f = \sum _{ k \in \mathbb Z} \phi_k (\sqrt{A_0}) f 
\quad \text{in } \mathcal Z_0 ' (\Omega) 
\end{equation}
by Lemma \ref{lem:decomposition1} {\rm (ii)}. 
Plugging \eqref{0128-1} into \eqref{EQ:formal-g}, we can write formally 
$$
_{\mathcal Z'_V}\langle f, g 
 \rangle _{\mathcal Z_V} 
 = \sum _{j \in \mathbb Z} \sum _{ k \in \mathbb Z} \, _{\mathcal Z'_V} 
 \langle \phi_k (\sqrt{A_0}) f, \phi_j (\sqrt{A_V}) g 
 \rangle _{\mathcal Z_V} . 
$$
Then it is sufficient to show that for any $g \in \mathcal Z_V (\Omega)$
\begin{equation}\label{1120-1}
\sum _{j \in \mathbb Z} \sum _{ k \in \mathbb Z} 
 \big| \, _{\mathcal Z'_V}\langle \phi_k (\sqrt{A_0}) f, \phi_j (\sqrt{A_V})g 
 \rangle _{\mathcal Z_V} \big| 
 \leq C \| f \|_{\dot B^s_{p,q} (A_0)} \| g \|_{\dot B^{-s}_{p',q'} (A_V)}  , 
\end{equation} 
since 
$$
\mathcal Z _V (\Omega) \hookrightarrow \dot B^{-s}_{p',q'} (A_V) .
$$
Let $\Phi_j := \phi_{j-1} + \phi_j + \phi_{j+1}$. 
By using $\phi_j = \phi_j \Phi_j$ and 
H\"older's inequality we estimate
\begin{align}\label{1020-7}
& 
\sum _{j \in \mathbb Z} \sum _{ k \in \mathbb Z} 
 \big| \, _{\mathcal Z_V'}\langle \phi_k (\sqrt{A_0}) f, \phi_j (\sqrt{A_V})g 
 \rangle _{\mathcal Z_V} \big|
\\
=&
\sum _{j \in \mathbb Z} \sum _{ k \in \mathbb Z} 
 \big| \, _{\mathcal Z_V'}\langle \phi_j (\sqrt{A_V}) \phi_k (\sqrt{A_0}) f  , 
                         \Phi_j (\sqrt{A_V})g \rangle _{\mathcal Z_V} 
 \big|
\nonumber \\
\leq & 
\Big\{ \sum _{j \in \mathbb Z} 
        \Big( 2^{sj}  \sum _{ k \in \mathbb Z} 
             \big\| \phi_j (\sqrt{A_V}) \Phi_k ( \sqrt{A_0})\phi_k (\sqrt{A_0}) f \big\|_{L^p}\Big)^q
     \Big\}^{\frac{1}{q}}
\nonumber \\
& \quad \times \nonumber
\Big\{ \sum _{j \in \mathbb Z} 
        \Big( 2^{-sj} 
             \big\| \Phi_j (\sqrt{A_V}) g 
             \big\|_{L^{p'}}
        \Big)^{q'}
     \Big\}^{\frac{1}{q'}}\\ 
     =: & I(s,f) \times II(s,g). 
\nonumber
\end{align}
The estimate of the second factor $II(s,g)$ is an immediate consequence of 
the definition of norm of Besov spaces $\dot B^{-s}_{p',q'} (A_V)$, that is, we have 
\begin{equation}\label{EQ:0g}
II(s,g) \leq C \| g \|_{\dot B^{-s}_{p',q'} (A_V)}. 
\end{equation}
As to the first factor $I(s,f)$, 
applying \eqref{1020-1}, 
we have, for any $j \in \mathbb Z$
\[
             \big\| \phi_j (\sqrt{A_V}) \Phi_k ( \sqrt{A_0})\phi_k (\sqrt{A_0}) f 
             \big\|_{L^p}
             \leq C 
\begin{cases}
2^{-\alpha (j-k)} \| \phi_k (\sqrt{A_0}) f  \|_{L^p} 
& \quad \text{if } k \leq j,
\\
 \| \phi_k (\sqrt{A_0}) f  \|_{L^p} 
& \quad \text{if } k \geq j ,
\end{cases}
\]
where $\alpha$ is a fixed constant such that $s < \alpha < \min \{ 2, n/p \}$. 
For the sake of simplicity, we 
put 
\begin{equation}\label{a_k}
a_k := \| \phi_k (\sqrt{A_0}) f  \|_{L^p}.
\end{equation}
When $k \leq j$, by using the above estimate, we estimate the first factor $I(s,f)$
in \eqref{1020-7} as 
\begin{equation}\label{1020-8}
\begin{split}
I(s,f)\leq & C \Big\{ \sum _{ j \in \mathbb Z} 
          \Big( 2^{sj} \sum _{k \leq j} 2^{-\alpha (j-k)} a_k \Big) ^q 
       \Big\}^{\frac{1}{q}}\\
=
&
C \Big\{ \sum _{ j \in \mathbb Z} 
          \Big( \sum _{k' \geq  0} 2^{-(\alpha -s) k' } 2^{s(j-k')}a_{j-k'} \Big) ^q 
       \Big\}^{\frac{1}{q}}
\\
\leq &
C \sum _{k' \geq  0} 2^{-(\alpha -s) k'} 
       \Big\{ \sum _{ j \in \mathbb Z} 
          \Big(  2^{s(j-k')}a_{j-k'} \Big) ^q 
       \Big\}^{\frac{1}{q}}
\\
\leq 
&
 C \| f \|_{\dot B^s_{p,q}}, 
\end{split}
\end{equation}
and when $k \geq j$, we have
\begin{equation}\label{1020-9}
\begin{split}
I(s,f)\leq& C \Big\{ \sum _{ j \in \mathbb Z} 
          \Big( 2^{sj} \sum _{k \geq j}  a_k \Big) ^q 
       \Big\}^{\frac{1}{q}}\\
=
&
  C \Big\{ \sum _{ j \in \mathbb Z} 
          \Big( \sum _{k' \leq  0} 2^{s k'} 2^{s(j-k')}a_{j-k'} \Big) ^q 
       \Big\}^{\frac{1}{q}}
\\
\leq
&
    C \sum _{k' \leq  0} 2^{s k'} 
       \Big\{ \sum _{ j \in \mathbb Z} 
          \Big(  2^{s(j-k')}a_{j-k'} \Big) ^q 
       \Big\}^{\frac{1}{q}}
\\
\leq
&
  C \| f \|_{\dot B^s_{p,q}}. 
\end{split}
\end{equation}
Summarizing \eqref{EQ:0g}--\eqref{1020-9},
we conclude that the series \eqref{EQ:formal-g}
is absolutely convergent, and hence, 
the identity \eqref{908-5} is justified. 
Also, as a consequence of 
\eqref{1020-8} and \eqref{1020-9}, 
we obtain  
\begin{align*} 
\| f \|_{\dot B^s_{p,q}(A_V)} 
\leq& \Big\{ \sum _{j \in \mathbb Z} 
        \Big( 2^{sj}  \sum _{ k \in \mathbb Z} 
             \big\| \phi_j (\sqrt{A_V}) \phi_k (\sqrt{A_0}) f 
             \big\|_{L^p}
        \Big)^q
     \Big\}^{\frac{1}{q}}\\
\leq& C \| f \|_{\dot B^s_{p,q}(A_0)} .
\end{align*}
Therefore, the embedding \eqref{EQ:Bes-1} holds.

It is also possible to show the embedding 
$$
\dot B^s_{p,q} (A_V) \hookrightarrow \dot B^s_{p,q} (A_0)
$$
by the same argument as above, if we 
apply \eqref{1020-2} instead of \eqref{1020-1}. 
The proof of isomorphism 
\eqref{EQ:Finalle} for $s > 0$ 
is complete. \\

\noindent 
{\bf The case: $s < 0$. }
In this case, the argument for 
$s>0$ works well. The only difference is to 
obtain estimates corresponding to 
\eqref{1020-8} and \eqref{1020-9}, 
so that we concentrate on proving that 
\begin{equation}\label{1020-10}
\Big\{ \sum _{j \in \mathbb Z} 
        \Big( 2^{sj}  \sum _{ k \in \mathbb Z} 
             \big\| \phi_j (\sqrt{A_V}) \Phi_k ( \sqrt{A_0})\phi_k (\sqrt{A_0}) f 
             \big\|_{L^p}
        \Big)^q
     \Big\}^{\frac{1}{q}}
 \leq C \| f \|_{\dot B^s_{p,q} (A_0)}. 
\end{equation}
It follows from \eqref{1020-1} that for any $j \in \mathbb Z$
\[
\big\| \phi_j (\sqrt{A_V}) \Phi_k ( 
\sqrt{A_0})\phi_k (\sqrt{A_0}) f 
             \big\|_{L^p}
\leq C 
\begin{cases}
 \| \phi_k (\sqrt{A_0}) f  \|_{L^p} 
& \quad \text{if } k \leq j , 
\\
2^{-\alpha (k-j)} \| \phi_k (\sqrt{A_0}) f  \|_{L^p} 
& \quad \text{if } k \geq j ,
\end{cases}
\]
where $\alpha$ is a fixed constant such that $|s| < \alpha < \min \{ 2, n(1-1/p) \}$. 
Then, by using the above estimate and recalling the definition \eqref{a_k} of $a_k$, we have 
for $k \leq j$, 
\begin{equation}\notag 
\begin{split}
& 
\Big\{ \sum _{j \in \mathbb Z} 
        \Big( 2^{sj}  \sum _{ k \leq j} 
             \big\| \phi_j (\sqrt{A_V}) \Phi_k ( \sqrt{A_0})\phi_k (\sqrt{A_0}) f 
             \big\|_{L^p}
        \Big)^q
     \Big\}^{\frac{1}{q}}
\\
\leq & C \Big\{ \sum _{ j \in \mathbb Z} 
          \Big( 2^{sj} \sum _{k \leq j} a_k 
          \Big) ^q 
       \Big\}^{\frac{1}{q}}
       \\
   =
   & 
   C \Big\{ \sum _{ j \in \mathbb Z} 
          \Big( 2^{sj} \sum _{k' \geq  0}  a_{j-k'} \Big) ^q 
       \Big\}^{\frac{1}{q}}
\\
=& C \Big\{ \sum _{ j \in \mathbb Z} 
          \Big( \sum _{k' \geq  0} 2^{s k'} 2^{s(j-k')}a_{j-k'} \Big) ^q 
       \Big\}^{\frac{1}{q}}
\\
\leq 
&
 C \sum _{k' \geq  0} 2^{s k} 
       \Big\{ \sum _{ j \in \mathbb Z} 
          \Big(  2^{s(j-k')}a_{j-k'} \Big) ^q 
       \Big\}^{\frac{1}{q}}
\\
\leq & C \| f \|_{\dot B^s_{p,q}(A_0)}, 
\end{split}
\end{equation}
and in the case when $k \geq j$, we estimate 
\begin{equation}\notag 
\begin{split}
& 
\Big\{ \sum _{j \in \mathbb Z} 
        \Big( 2^{sj}  \sum _{ k \geq j} 
             \big\| \phi_j (\sqrt{A_V}) \Phi_k (\sqrt{A_0})\phi_k (\sqrt{A_0}) f 
             \big\|_{L^p}
        \Big)^q
     \Big\}^{\frac{1}{q}}
\\
\leq & C \Big\{ \sum _{ j \in \mathbb Z} 
          \Big( 2^{sj} \sum _{k \geq j} 2^{-\alpha (k-j)} a_k \Big) ^q 
       \Big\}^{\frac{1}{q}}
\\
   = 
   &
    C \Big\{ \sum _{ j \in \mathbb Z} 
          \Big( 2^{sj} \sum _{k' \leq  0} 2^{\alpha k'} a_{j-k' } \Big) ^q 
       \Big\}^{\frac{1}{q}}
\\
=& C \Big\{ \sum _{ j \in \mathbb Z} 
          \Big( \sum _{k' \leq  0} 2^{(\alpha + s) k'} 2^{s(j-k')}a_{j-k'} \Big) ^q 
       \Big\}^{\frac{1}{q}}
\\
\leq 
&
C \sum _{k' \leq  0} 2^{(\alpha + s) k'} 
       \Big\{ \sum _{ j \in \mathbb Z} 
          \Big(  2^{s(j-k')}a_{j-k'} \Big) ^q 
       \Big\}^{\frac{1}{q}}
\\
\leq & C \| f \|_{\dot B^s_{p,q}(A_0)}. 
\end{split}
\end{equation}
Therefore, the estimate \eqref{1020-10} is verified, and the proof of 
the isomorphism 
\eqref{EQ:Finalle} for $s < 0$ 
is finished.\\

\noindent 
{\bf The case: $s = 0$. }
In this case
we have only to show the corresponding 
estimates to \eqref{1020-10}.
Since $1 < p < \infty$, Lemma \ref{lem:1020-2} implies that 
\[ 
\| \phi_j (\sqrt{A_V}) \Phi_k (\sqrt{A_0}) f \|_{L^p} 
\leq C 2^{-\alpha |j-k|} \| f \|_{L^p} , 
\]
\[
\| \phi_k (\sqrt{A_0}) \Phi_j (\sqrt{A_V})  f \|_{L^p} 
\leq C 2^{-\alpha |j-k|} \| f \|_{L^p} , 
\]
where $0 < \alpha < \min \{2, n/p , n (1-1/p) \}$. 
Then it follows from 
Young's inequality that 
\begin{equation}\notag 
\begin{split}
& 
\Big\{ \sum _{j \in \mathbb Z} 
        \Big(  \sum _{ k \in \mathbb Z} 
             \big\| \phi_j (\sqrt{A_V}) \Phi_k ( \sqrt{A_0 })\phi_k (\sqrt{A_0}) f 
             \big\|_{L^p}
        \Big)^q
     \Big\}^{\frac{1}{q}}
\\
\leq & 
C
\Big\{ \sum _{j \in \mathbb Z} 
        \Big(  \sum _{ k \in \mathbb Z} 
             2^{-\alpha |j-k|} \big\| \phi_k (\sqrt{A_0}) f 
             \big\|_{L^p}
        \Big)^q
     \Big\}^{\frac{1}{q}}
\\
\leq 
& 
C \Big( \sum _{j \in \mathbb Z} 2^{-\alpha |j|} \Big)
     \Big\{ \sum _{ k \in \mathbb Z} 
           \big\| \phi_k (\sqrt{A_0}) f \big\|_{L^p} ^q 
     \Big\} ^{\frac{1}{q}}
\\
\leq & 
C \| f \|_{\dot B^0_{p,q}(A_0)}. 
\end{split}
\end{equation}
Therefore, the case $s = 0$ also holds.
Thus the proof of isomorphism \eqref{EQ:Finalle} 
for homogeneous case is finished.
\hfill $\square$

\vskip3mm 

Let us now prove the inhomogeneous case.\\

\noindent
{\bf Proof of the isomorphism:
\begin{equation}\label{EQ:Fin-Finale} 
 B^s_{p,q} (A_0) \cong B^s_{p,q} (A_V). 
 \end{equation}}
 
The proof of \eqref{EQ:Fin-Finale} 
is similar to the homogeneous case.
Indeed, as to the proof of the 
embedding 
$$B^s_{p,q} (A_0) \hookrightarrow 
B^s_{p,q} (A_V),$$ 
the main point is to show that for any $f \in B^s_{p,q} (A_0)$, 
\begin{equation}\notag 
f = \psi (A_V) f + \sum _{ j \in \mathbb N} \phi_j (\sqrt{A_V}) f 
 \quad \text{in } \mathcal X'_V (\Omega) . 
\end{equation}
This identity is obtained by using the following estimate: 
\begin{equation}\notag 
\begin{split}
& 
 \big| \, _{\mathcal X'_V}\langle \psi (A_0) f, \psi (A_V) )g 
 \rangle _{\mathcal X_V} \big| 
+ \sum _{j \in \mathbb N} \sum _{ k \in \mathbb N} 
 \big| \, _{\mathcal X'_V}\langle \phi_k (\sqrt{A_0}) f, \phi_j (\sqrt{A_V})g 
 \rangle _{\mathcal X_V} \big| 
\\
& 
+ \sum _{j \in \mathbb N} 
 \big| \, _{\mathcal X'_V}\langle \psi_k (A_0) f, \phi_j (\sqrt{A_V})g 
 \rangle _{\mathcal X_V} \big| 
+ \sum _{ k \in \mathbb N} 
 \big| \, _{\mathcal X'_V}\langle \phi_k (\sqrt{A_0}) f, \psi (A_V) )g 
 \rangle _{\mathcal X_V} \big| 
\\
 \leq 
 & C \| f \|_{B^s_{p,q} (A_V)} \| g \|_{B^{-s}_{p',q'} (A_V)}  
\end{split}
\end{equation} 
 for any $g \in \mathcal X_V (\Omega)$. 
The proof of 
the above estimate is analogous to those of \eqref{1020-8} and \eqref{1020-9} 
by taking the sum over $j,k \in \mathbb N$. So we may omit the details. 
Thus we conclude 
\eqref{EQ:Fin-Finale}.
\hfill $\square$


\appendix
 \section{($L^p$-boundedness, self-adjointness and pointwise estimates for $e^{-tA_V}$)} 
 \label{App:AppendixA}

We discuss the uniform $L^p$-boundedness of $\phi (\theta A_V)$ in this appendix. 

\begin{prop}\label{prop:Lp-bound}
Let $\phi \in \mathcal{S}(\mathbb{R})$ and $1 \le p \le \infty$. 
\begin{enumerate}
\item[(i)] Assume that $V$ satisfies \eqref{1104-1}. Then
\begin{equation}\label{bd:inh}
 \sup _{ 0 < \theta \leq 1} 
 \| \phi (\theta A_V) \|_{L^p \to L^p} 
 < \infty . 
\end{equation}

\item[(ii)] 
Assume that $V$ satisfies \eqref{1104-1} and \eqref{ass:1}. 
Then
\begin{equation}\label{bd:hom}
 \sup _{ 0 < \theta < \infty} 
 \| \phi (\theta A_V) \|_{L^p \to L^p} 
 < \infty . 
\end{equation}
\end{enumerate}
\end{prop}

\noindent 
{\bf Remark.} 
\noindent 
We note that the potential like 
\begin{equation}\notag 
V(x) \simeq - c |x|^{-2} \quad \text{as } |x | \to \infty , 
\quad c > 0
\end{equation}
is very interesting. However, it 
is excluded from assumption \eqref{ass:1} on $V$. 
The reason is that the uniform boundedness in Proposition~\ref{prop:Lp-bound} 
would not be generally obtained, 
since 
$$
\lim _{t \to \infty}\| e^{-t A_V} \|_{L^p \to L^p} = \infty
$$
for some $p \not = 2$ which was proved in \cite{IIY-2013,IIY-2015}.

\vskip3mm

The proof of Proposition \ref{prop:Lp-bound} is similar to that of 
our previous works \cite{IMT-preprint,IMT-ISAAC} 
by 
using Lemmas \ref{lem:s.a.}--\ref{lem:pointwise} below 
and we may omit the complete proof of Proposition \ref{prop:Lp-bound}.
So, we shall concentrate on the proof of the self-adjointness of 
$A_V$ and the pointwise estimate of integral kernel of 
$e^{-tA_V}$, which need certain adjustment to the method 
in \cite{IMT-ISAAC}.
\\

Let us prove that $A_V$ is self-adjoint on $L^2 (\Omega)$. 
Following the argument in \cite{ReeSim_1972} (see also \cite{IMT-preprint}), 
we 
consider the quadratic form $q$ 
defined by letting
$$
q(u,v):=
\int_{\Omega} \nabla u(x)\cdot \overline{\nabla v(x)} \, dx 
 + \int_{\Omega} V(x)u(x)\overline{v(x)} \, dx, \quad u,v\in \mathcal{Q}(q),
$$
where $\mathcal Q(q) := 
\{ u \in H^1_0 (\Omega) \, | \, \sqrt{V_+} u \in L^2 (\Omega) \}$. 

\begin{lem}\label{lem:s.a.}
Assume that the measurable potential 
$V$ satisfies \eqref{1104-1}. 
Then there exists a self-adjoint operator $A_V$ on $L^2 (\Omega)$ 
such that 
\begin{equation}\notag 
\begin{cases}
\mathcal{D}(A_V)
= \left\{ u \in \mathcal{Q}(q)\, \Big| \, 
  ^\exists w_u \in L^2 (\Omega) \text{ s.t. } 
  \displaystyle \int_{\Omega} w_u \overline{v} \,dx = q(u,v)
  \text{ for all } v\in \mathcal Q (q) \right\},
\\
A_V u = w_u  \quad \text{for } u \in\mathcal{D}(A_V) . 
\end{cases}
\end{equation}
Moreover, $A_V$ is semi-bounded, i.e., 
there exists a constant $\lambda _ 0 > 0$ such that 
$$A_V \geq - \lambda _0 ^2 I.$$
\end{lem}

\noindent 
{\bf Remark. } 
We define $A_V u \in L^2 (\Omega)$  for $u \in \mathcal Q(q) $ 
if $u \in \mathcal D (A_V)$. Then $\mathcal D (A_V)$ is simply rewritten as 
\begin{equation}\notag 
\begin{split}
\mathcal D (A_V) 
& 
= \{ u \in \mathcal Q (q) \, | \, A_V u \in L^2 (\Omega)\}
\\
& 
= \{ u \in H^1_0 (\Omega) \, | \, \sqrt{V_+} u \in L^2 (\Omega), A_V u \in L^2 (\Omega) \}. 
\end{split}
\end{equation}
This is nothing but the identity 
\eqref{EQ:ID} given in \S 1. 
\\

To prove Lemma \ref{lem:s.a.}, we need the following lemma.

\begin{lem}\label{lem:0215-1}
{\rm (}\cite{DP-2005,IMT-preprint,Simon-1982}{\rm)}
Assume that $V_- $ is in the Kato class $K_n (\Omega)$. 
Then for any $\varepsilon > 0$ there exists $\lambda_0 > 0$ such that 
\begin{equation}\label{0215-1}
\int_{\Omega} V_- (x) |f (x)|^2 \, dx 
\leq \varepsilon \| \nabla f \|_{L^2} ^2 
   + \lambda _0 ^2 \| f \|_{L^2} ^2 
\end{equation}
for any $f \in H^1_0 (\Omega)$. 
\end{lem}

\noindent 
{\bf Proof. } 
By the density argument, we may take $f \in C_0^\infty (\Omega)$. 
Let 
$\tilde{f}$ and $\tilde{V}_{-}$ be the 
zero extension of $f$ and $V_{-}$ to 
$\mathbb R^n$, respectively. 
Then \eqref{0215-1} is equivalent to 
$$
\int_{\mathbb R^n} \tilde V_- (x) |\tilde f (x)|^2 \, dx 
\leq \varepsilon \| \nabla \tilde f \|_{L^2 (\mathbb R^n)} ^2 
   + \lambda _0 ^2 \| \tilde f \|_{L^2 (\mathbb R^n)} ^2 . 
$$
Since this inequality is proved in \cite{DP-2005,Simon-1982}, we may omit the details. 
The proof of Lemma~\ref{lem:0215-1} is finished. 
\hfill $\square$

\begin{proof}
[\bf Proof of Lemma \ref{lem:s.a.}] 
It suffices to show that the quadratic form $q$ is closed and semi-bounded 
by Theorem VIII.15 in \cite{ReeSim_1972} (see also Lemma 2.3 in \cite{IMT-preprint}). 
We first show that $q$ is closed. Put
\[
q_1(u,v)=\int_{\Omega} \nabla u(x)\cdot \overline{\nabla v(x)} \, dx 
- \int_{\Omega} V_-(x)u(x)\overline{v(x)} \, dx, \quad u,v\in \mathcal{Q}(q_1):=H^1_0(\Omega),
\]
\[
q_2(u,v)=\int_{\Omega} V_+(x)u(x)\overline{v(x)} \, dx, \quad 
u,v\in \mathcal{Q}(q_2) :=\{u\in L^2(\Omega)\, |\, \sqrt{V_+}u\in L^2(\Omega)\}.
\]
Then we get 
\[
q(u,v)=q_1(u,v)+q_2(u,v), \quad u,v\in \mathcal{Q}(q)=\mathcal{Q}(q_1)\cap\mathcal{Q}(q_2).
\]
Since $q_1$ is closed (see Proposition 2.1 in \cite{IMT-preprint}) and 
the sum of two closed quadratic forms is also closed, 
it is enough to show that $q_2$ is closed. 

We show that $q_2$ is closed. 
Put $q_2(u)=q_2(u,u)$ for simplicity. 
Assume that 
\[
u\in L^2(\Omega), \quad u_j\in\mathcal{Q}(q_2), \quad 
q_2(u_j-u_k)\to 0, \quad \|u_j-u\|_{L^2}\to0 
\quad \text{as } j,k\to\infty, 
\]
and we prove that 
\begin{equation}\label{1113-1}
u\in\mathcal{Q}(q_2) \quad \text{and} \quad q_2(u_j-u)\to0\quad \text{as } j \to \infty.
\end{equation}
Since $\{\sqrt{V_+}u_j\}_{j=1}^\infty$ is a Cauchy sequence in $L^2(\Omega)$, there exists $v\in L^2(\Omega)$ such that 
\[
\sqrt{V_+}u_j\to v\quad \text{in } L^2(\Omega). 
\]
Hence the sequence $\{\sqrt{V_+}u_j\}_{j=1}^\infty$ converges to $v$ almost everywhere 
along a subsequence denoted by the same, namely, 
$$
\sqrt{V_+}u_j (x) \to v (x) \quad \text{for almost every } x \in \Omega 
\text{ as } j \to \infty . 
$$
On the other hand, since any convergent sequence in $L^2 (\Omega)$ contains 
a subsequence which converges almost everywhere in $\Omega$, it follows that 
\[
\sqrt{V_+}u_j (x) \to \sqrt{V_+}u (x) \quad \text{for almost every } x \in \Omega  
\text{ as } j \to \infty .  
\]
Summarizing three convergences obtained now, 
we get $\sqrt{V_+} u = v \in L^2 (\Omega)$. This proves \eqref{1113-1}.

Finally we prove that $q$ is semi-bounded. Let $u\in\mathcal{D}(A_V)$. 
By the estimate \eqref{0215-1}, 
for any $\varepsilon > 0$ there exists $\lambda_0 > 0$ such that 
\begin{equation}\notag 
\begin{split}
\int_{\Omega} (A_V u) \overline{u} \, dx 
&\geq \int_{\Omega} |\nabla u|^2dx 
 - \int_{\Omega}V_-|u|^2\,dx\\
&\ge\|\nabla u\|^2_{L^2}-\varepsilon\|\nabla u\|^2_{L^2}-\lambda_0^2 \|u\|^2_{L^2}.
\end{split}
\end{equation}
Taking $\varepsilon=1$, we obtain
\begin{equation}
\notag 
\int_{\Omega} (A_V u) \overline{u} \, dx 
\geq -\lambda_0^2 \|u\|^2_{L^2(\Omega)},
\end{equation}
which implies that $q$ is semi-bounded.
This ends the proof of Lemma \ref{lem:s.a.}.
\end{proof}

As to the pointwise estimate on the kernel of $e^{-tA_V}$, 
we have the following. 

\begin{lem}\label{lem:pointwise}
The integral kernel $e^{-tA_V} (x,y)$
of the semi-group $\{e^{-tA_V}\}_{t \geq 0}$ 
enjoys 
the following estimates{\rm :}
\begin{enumerate}
\item[(i)] Assume that $V$ satisfies \eqref{1104-1}. 
Then there exist constants 
$ \omega, C > 0$ such that 
\begin{equation}\notag 
|e^{-tA_V} (x,y)| 
\leq C e^{\omega t} t^{-\frac{n}{2}} \exp \Big( -\frac{|x-y|^2}{Ct} \Big)  
\end{equation}
for any $t>0$ and $x,y\in\Omega$.
In particular, we have
\begin{equation}\label{1117-4_2}
|e^{-tA_V} (x,y)| 
\leq C t^{-\frac{n}{2}} \exp \Big( -\frac{|x-y|^2}{Ct} \Big) 
\quad \text{if }  \,\, 0 < t \leq 1.
\end{equation}

\item[(ii)] Assume that $V$ satisfies \eqref{1104-1} and \eqref{ass:1}. 
Then there exists $ C > 0$ such that 
\begin{equation}\label{1117-5}
|e^{-tA_V} (x,y)| 
\leq C t^{-\frac{n}{2}} \exp \Big( -\frac{|x-y|^2}{Ct} \Big)  
\end{equation}
for any $t>0$ and $x,y\in\Omega$.
\end{enumerate}
\end{lem}

\begin{proof}
[\bf Proof] 
Since the assertions {\rm (i)} with $n \geq 1$ and {\rm (ii)} with $n = 1,2$ were already proved 
in \cite{O-2006},  
it is enough to show {\rm (ii)} in the case when $n \geq 3$. 
We prove the lemma for $n \geq 3$ in a formal way for the sake of simplicity. 
For more rigorous argument, see \cite{IMT-preprint,IMT-ISAAC}. 

Put 
$$V_* := - V_- . 
$$
It is proved in Proposition 3.1 from \cite{IMT-preprint} that if $V_*$ 
satisfies assumption (2.3), then 
\begin{equation} \notag 
 |e^ {-t A_{V_*}} (x,y)| 
 \leq C t^{-\frac{n}{2}} e^{-\frac{|x-y|^2}{Ct}} 
 \quad \text{for } t > 0 , \quad  x , y \in \Omega . 
\end{equation}
Once the following inequality 
\begin{equation}\label{129-2}
|e^{-t A_{V}} (x,y)| 
\leq |e^{-t A_{V_*}} (x,y)|
\end{equation}
is proved, the proof of the lemma is complete. 
So, we prove \eqref{129-2}. 
Let 
\begin{gather}
\notag 
u^{(1)} (t) := e^{-tA_{V_*}} f, \quad u^{(2)} (t) := e^{-tA_{V}} f, 
\\ \notag 
u (t) := u^{(1)} (t) - u^{(2)} (t), \quad 
u_- (t) := - \min \{  u (t), 0 \}, 
\end{gather}
where $f \in C_0 ^\infty (\Omega)$ is non-negative. 
Note that $u^{(1)} (t) \in \mathcal D (A_{V_*})$ and $u^{(2)} (t) \in \mathcal D(A_V)$ 
for any $t > 0$.  
An explicit calculation and 
non-negativity of $A_{V_*}$ imply 
 that 
\begin{equation}\notag 
\begin{split}
\frac{1}{2}\frac{d}{dt} \int_{\Omega} \big( u_- \big) ^2 \, dx 
& 
= - \int_{\Omega} (\partial_t u)  u_- \, dx
\\
& 
= \int_{\Omega} \big( A_{V_*} u  \big) u_- \, dx 
  - \int_{\Omega} V_+ u^{(2)} (t) u_-  \, dx  
\\
& 
= - \int_{\Omega} \big( |\nabla u_-|^2 - V_- (u_{-} )^2 \big)\, dx 
  - \int_{\Omega} V_+ u^{(2)} (t) u_-  \, dx  
\\
& \leq - \int_{\Omega} V_+ u^{(2)} (t) u_-  \, dx .
\end{split}
\end{equation}
For the negative part of $u^{(2)}$, i.e., 
$$
u^{(2)}_- (t):= - \min\{ u^{(2)} (t), 0 \},
$$
it follows from the analogous argument to the above that 
\begin{equation}\notag 
\begin{split}
\frac{1}{2}\frac{d}{dt} \int_{\Omega} \big( u^{(2)}_- \big) ^2 \, dx 
& 
= \int_{\Omega} \big( A_V u^{(2)} \big) u^{(2)}_- \, dx 
\\
& 
= - \int_{\Omega} \big ( |\nabla u^{(2)}_-| ^2 - V_- (u^{(2)}_-)^2 \big) 
   \, dx
   - \int_{\Omega} V_+ \big ( u^{(2)}_2 \big) ^2 \, dx
\\
& \leq 0 . 
\end{split}
\end{equation}
The above two inequalities imply that 
$\| u_- (t) \|_{L^2} ^2$ and $\| u^{(2)}_- (t) \|_{L^2} ^2$ do not increase.  
Hence we conclude that
$$
u_- (t) = u^{(2)}_- (t) = 0  \quad \text{in } L^2 (\Omega)
$$
for all $ t \geq 0$,  since 
$u_- (0) = 0$, $u^{(2)}_- (0) = 0$. 
Therefore, we get 
$$
0 \leq e^{-tA_{V}} f \leq  e^{-tA_{V_*}} f. 
$$
For each point $x_0 \in\Omega$, by taking $f = f_k$ ($k = 1,2,\cdots$) 
which tend to 
the delta function supported at $x_0$, 
we see that the kernel of $e^{-tA_{V}} $ is bounded by that of 
$e^{-tA_{V_*}} $. 
Thus \eqref{129-2} is proved. 
This ends the proof of Lemma \ref{lem:pointwise}.
\end{proof}

\section{}
In this appendix we prove that zero is not an eigenvalue of $A_V$.  
\label{App:AppendixB}

\begin{lem}\label{lem:zero}
Assume that $V$ satisfies \eqref{1104-1} and 
\begin{gather}
\label{ass:B_1}
\begin{cases} 
V _- = 0
& \quad \text{if } n = 1,2, 
\\
\displaystyle \sup _{x \in \Omega} 
   \int_{\Omega} \dfrac{|V_- (y)|}{|x-y|^{n-2}} \, dy 
 < \dfrac{ 4 \pi^{\frac{n}{2}}}{\Gamma (n/2 -1)}
& \quad \text{if } n \geq 3 , 
\end{cases}
\end{gather}
where $\Gamma (\cdot)$ is the Gamma function. 
Then $A_V$ is non-negative on $L^2 (\Omega)$, and 
zero is not an eigenvalue of $A_V$. 
\end{lem}

For the difference between assumptions \eqref{ass:1} and \eqref{ass:B_1}, 
we refer to Proposition~\ref{prop:Lp-bound} and Lemma~\ref{lem:pointwise} {\rm (ii)} 
(cf. Proposition~3.1 in \cite{IMT-preprint} and Proposition~5.1 in \cite{DP-2005}). 
\\

To prove Lemma \ref{lem:zero}, we need the following lemma.

\begin{lem}\label{lem:V_-} 
{\rm (}\cite{DP-2005,IMT-preprint}{\rm)} 
Assume that $n \geq 3$. Suppose that $V_-$ satisfies 
$$
\| V_- \|_{K_n (\Omega)} 
:= \sup_{x \in \Omega} \int_{\Omega} \frac{| V_-(y)|}{|x-y|^{n-2}} \,dy < \infty. 
$$
Then 
\begin{equation}\label{EQ:V-}
 \int_{\Omega} V_- (x) |f(x)|^2 \, dx 
\leq \frac{\Gamma(n/2 - 1) \|V_-\|_{K_{n}(\Omega)}}{4\pi^{n/2}}
  \|\nabla f\|^2_{L^2(\Omega)}
\end{equation}
for any $f \in H^1_0 (\Omega)$. 
\end{lem}

The proof is similar to Lemma~\ref{lem:0215-1}. 
So we may omit the details.

\begin{proof}
[\bf Proof of Lemma \ref{lem:zero}]
First we consider the case when $n \geq 3$. 
We prove that 
if $f$ satisfies 
\begin{equation}\label{0216-1}
f \in \mathcal D (A_V) \quad \text{and} \quad 
A_V f = 0 \text{ in } L^2 (\Omega), 
\end{equation}
then $f = 0$. 
Indeed, we find from  
Lemma~\ref{lem:V_-} and assumption \eqref{0216-1} that 
\begin{align} \label{0219-1}
0 = &\int_{\Omega} (A_V f ) \overline {f} \, dx
= \int_{\Omega} \big( |\nabla f|^2 - V_- |f|^2 \big) \, dx 
 +\int_{\Omega} V_+ |f|^2 \, dx 
 \\ \notag 
\geq& \left( 1- \frac{\Gamma(n/2 - 1) \|V_-\|_{K_{n}(\Omega)}}{4\pi^{n/2}} 
      \right)
      \| \nabla f \|_{L^2} ^2,
\end{align}
which implies that $f=0$, since 
$f\in\mathcal{D}(A_V) \subset H^1_0(\Omega)$. 
Finally, the above inequality also implies that $A_V$ is non-negative on $L^2 (\Omega)$. 
When $n = 1,2$, it follows that $V_- = 0$, and hence, 
we conclude from \eqref{0219-1} that $\| \nabla f \|_{L^2} = 0$. 
Thus we get $f = 0$ also in this case. 
This ends the proof of Lemma~\ref{lem:zero}. 
\end{proof}

%

\end{document}